\documentclass[article, 12pt, reqno]{amsart}
\usepackage[includeheadfoot, margin=1in]{geometry}
\usepackage{amsfonts}
\usepackage{amsmath}
\usepackage{amsthm}
\usepackage{amssymb}
\usepackage{color}
\usepackage[mathscr]{euscript}
\usepackage{bm}
\usepackage{amscd}
\usepackage{comment}
\usepackage[numbers,square]{natbib}
\usepackage{hyperref}

\usepackage{xcolor}
\hypersetup{
    colorlinks,
    linkcolor={red!50!black},
    citecolor={blue!50!black},
    urlcolor={blue!80!black}
}

\usepackage{enumerate}
\usepackage{tikz-cd}
\usepackage{tikz}
\usetikzlibrary{graphs}
\usepackage{subcaption}
\usepackage{mathtools}

\numberwithin{equation}{section}

\newtheorem{Thm}{Theorem}[section]

\newtheorem{Lem}[Thm]{Lemma}

\newtheorem{Claim}[Thm]{Claim}

\theoremstyle{definition}
\newtheorem{Def}[Thm]{Definition}
\newtheorem{Ex}[Thm]{Example}

\theoremstyle{remark}
\newtheorem{Remark}[Thm]{Remark}

\newtheorem*{Ack}{Acknowledgements}

\begin{document}

\title{A Floer homology invariant for $3$-orbifolds via bordered Floer theory}
\author{Biji Wong}
\address{CIRGET, Universit\'{e} du Qu\'{e}bec \`{a} Montr\'{e}al, PO Box 8888, Station Centre-ville, Montr\'{e}al, Qu\'{e}bec H3C 3P8}
\email{biji.wong@cirget.ca}
\urladdr{https://sites.google.com/view/cirget-bijiwong}
\maketitle

\begin{abstract}
Using bordered Floer theory, we construct an invariant $\widehat{\mathit{HFO}}(Y^{\text{orb}})$ for $3$-orbifolds $Y^{\text{orb}}$ with singular set a knot that generalizes the hat flavor $\widehat{\mathit{HF}}(Y)$ of Heegaard Floer homology for closed $3$-manifolds $Y$. We show that for a large class of $3$-orbifolds $\widehat{\mathit{HFO}}$ behaves like $\widehat{\mathit{HF}}$ in that $\widehat{\mathit{HFO}}$, together with a relative $\mathbb{Z}_2$-grading, categorifies the order of $H_1^{\text{orb}}$. When $Y^{\text{orb}}$ arises as Dehn surgery on an integer-framed knot in $S^3$, we use the $\{-1,0,1\}$-valued knot invariant $\varepsilon$ to determine the relationship between $\widehat{\mathit{HFO}}(Y^{\text{orb}})$ and $\widehat{\mathit{HF}}(Y)$ of the $3$-manifold $Y$ underlying $Y^{\text{orb}}$.
\end{abstract}

\section{Introduction}
Heegaard Floer homology, introduced by Ozsv\'{a}th and Szab\'{o} in \cite{OS1}, is a package of invariants for closed $3$-manifolds that has produced a wealth of results in a variety of areas such as contact topology \cite{OS5, KMVW}, Dehn surgery \cite{Greene}, and knot theory \cite{OS3, OS4, OS7, Ni, Hom, OS6}. The purpose of this paper is to extend the hat version $\widehat{\mathit{HF}}$ of Heegaard Floer homology (with $\mathbb{Z}_2$ coefficients) to (orientable) $3$-orbifolds $Y^{\text{orb}}$ with singular locus a knot $K$. 

Three-orbifolds are spaces that locally look like quotients of $\mathbb{R}^3$ by finite subgroups of $SO(3)$. Over the past twenty years, much work has been done to construct homology invariants for $3$-orbifolds using gauge-theoretic ideas from Floer's original instanton homology theory \cite{Floer}, first by Collin and Steer in \cite{CS}, then by Kronheimer and Mrowka in \cite{KM1, KM2, KM3}. In this paper we offer up another homological invariant using the more combinatorial tool of bordered Heegaard Floer homology developed by Lipshitz, Ozsv\'{a}th, and D. Thurston in \cite{LOT1, LOT2} for $3$-manifolds with boundary. Specifically, we fix an equivariant neighborhood $N$ of the singular curve $K$ (together with some additional data for the equivariant torus boundary $\partial N$) and decompose the $3$-orbifold $Y^{\text{orb}}$ along $\partial N$. To $N$ we associate a (bounded) Type D structure that is sensitive to the equivariance around $K$. To the complement of $N$ (with induced data for its boundary) we associate the Type A structure given to us by bordered Floer theory. Motivated by the pairing theorem in bordered Floer theory, we define $\widehat{\mathit{HFO}}(Y^{\text{orb}})$ to be the homology of the box tensor product of the Type A structure with the Type D structure.

\begin{Thm}\label{definvariant}
$\widehat{\mathit{HFO}}(Y^{\emph{orb}})$ is a well-defined invariant of $Y^{\emph{orb}}$. Furthermore, when $Y^{\emph{orb}}$ is a $3$-manifold, $\widehat{\mathit{HFO}}(Y^{\emph{orb}})$ agrees with $\widehat{\mathit{HF}}(Y^{\emph{orb}})$.
\end{Thm}

The underlying space $\left|Y^{\text{orb}}\right|$ of any $3$-orbifold $Y^{\text{orb}}$ is a $3$-manifold in a natural way, so one might wonder how $\widehat{\mathit{HFO}}(Y^{\text{orb}})$ compares to $\widehat{\mathit{HF}}(\left|Y^{\text{orb}}\right|)$. When the $3$-orbifold comes from Dehn surgery on an integrally framed knot $K \subset S^3$, we prove that the difference between $\widehat{\mathit{HFO}}(Y^{\text{orb}})$ and $\widehat{\mathit{HF}}(\left|Y^{\text{orb}}\right|)$ depends on $3$ integers: the framing on $K$, the singular order around $K$, and the $\{-1,0,1\}$-valued knot invariant $\varepsilon(K)$ introduced by Hom in \cite{Hom}.

\begin{Thm}\label{relationshipHFhat}
Let $Y$ be $r$-surgery on a knot $K \subset S^3$ where $r$ is any integer. Let $Y^{\emph{orb}}$ be the 3-orbifold with underlying space $Y$ and singular curve $K$ of order $n$. If $\varepsilon (K) = 0$ and $r = 0$, then $\emph{rank}\big(\widehat{\mathit{HFO}}(Y^{\emph{orb}})\big) = n \cdot \emph{rank}\big(\widehat{\mathit{HF}}(Y)\big) - 2n + 2$. Otherwise, $\emph{rank}\big(\widehat{\mathit{HFO}}(Y^{\emph{orb}})\big) = n \cdot \emph{rank}\big(\widehat{\mathit{HF}}(Y)\big)$.
\end{Thm}

\noindent As an example, take $r=0$ and $K$ the unknot. Then $Y = S^2 \times S^1$ and $\varepsilon(K)=0$. Theorem \ref{relationshipHFhat} tells us that for every $n$, $\textrm{rank}\big(\widehat{\mathit{HFO}}(Y^{\textrm{orb}})\big)=2$.

For $3$-manifolds $Y$, it's well-known that $\widehat{\mathit{HF}}(Y)$ categorifies the order of $H_1(Y)$, see \cite{OS2}. We have an analogous result for a large class of $3$-orbifolds $Y^{\text{orb}}$:

\begin{Thm}\label{categorification}
There exists a relative $\mathbb{Z}_2$-grading on $\widehat{\mathit{HFO}}(Y^{\emph{orb}})$ so that if $Y^{\emph{orb}}$ has nullhomologous singular curve or comes from Dehn surgery on a framed knot in $S^3$, then up to sign $\chi\big(\widehat{\mathit{HFO}}(Y^{\emph{orb}})\big) = \left|H_1^{\emph{orb}}(Y^{\emph{orb}})\right|$.
\end{Thm}

Closely related to $\widehat{\mathit{HF}}$ is the plus version $\mathit{HF}^+$ of Heegaard Floer homology, and for $3$-manifolds $Y$ with $b_1 (Y) > 0$ it's known that $\mathit{HF}^+ (Y)$ categorifies the Turaev torsion invariant of $Y$ \cite{OS2}. Recently, the author extended the Turaev torsion invariant to $3$-orbifolds (with singular set a link) \cite{Wong}, so it is natural to ask if there is a homology theory for $3$-orbifolds generalizing $HF^+$ that categorifies this orbifold torsion invariant. The present paper can be thought of as a first step towards this goal.

Due to recent work of Hanselman, Rasmussen, and Watson \cite{HRW}, the bordered Floer invariants for $3$-manifolds with torus boundary can be thought of geometrically as decorated immersed curves on the punctured torus. Using this we get a geometric formulation of the orbifold homology invariant, the details of which will appear in a subsequent paper.

At the Perspectives in Bordered Floer Conference in May 2018, a connection between the orbifold invariant and Heegaard Floer with twisted coefficients was pointed out to the author by Matt Hedden and Adam Levine. This too will be written up in a later paper.

This paper is structured as follows. Section \ref{Sec: Background} collects the background on $3$-orbifolds, bordered Floer homology, and knot Floer homology that we will need, adapting some of it a bit to our situation. In Section \ref{Sec: Orbifold Invt} we define the orbifold invariant, prove Theorem \ref{definvariant}, and compute the invariant for several examples. In Section \ref{sec4} we prove Theorem \ref{relationshipHFhat} and give more examples. In Section \ref{sec5} we prove Theorem \ref{categorification}.

\begin{Ack} 
The author is grateful to Robert Lipshitz, Liam Watson, and Adam Levine for helpful conversations, and to Ina Petkova and Steve Boyer for encouragement and support.
\end{Ack}

\section{Background}\label{Sec: Background}

\subsection{3-orbifolds}\label{Sec:orbifoldsbackground}

Here we give a brief overiew of 3-orbifolds. For a more in-depth discussion, we refer the reader to \cite{Thurston, Scott, BMP, KL}. A \textit{3-orbifold} $Y^{\text{orb}}$ is a Hausdorff, second-countable space $\left|Y^{\text{orb}}\right|$ with an atlas $\{(U_i, \widetilde{U_i}, G_i, \phi_i)\}$ consisting of an open cover $\{U_i\}$ of $\left|Y^{\text{orb}}\right|$, connected and open sets $\widetilde{U_i} \subset \mathbb{R}^3$, continuous and effective actions of finite subgroups $G_i$ of $O(3)$ on $\widetilde{U_i}$, and homeomorphisms $\phi_i: \widetilde{U_i}/G_i \rightarrow U_i$. If $U_i \subset U_j$, then there is an injective homomorphism $f_{ji}: G_i \rightarrow G_j$ and a topological embedding $\widetilde{\phi}_{ji}:  \widetilde{U_i} \rightarrow  \widetilde{U_j}$, equivariant with respect to $f_{ji}$, that makes the following diagram commute:

\begin{center}
$\begin{CD}
\widetilde{U_i} @>\widetilde{\phi}_{ji}>> \widetilde{U_j}\\
@VVqV @VVqV\\
\widetilde{U_i}/G_i @>\phi_{ji}>> \widetilde{U_j}/G_j\\
@VV\phi_iV @VV\phi_jV\\
U_i @>incl>> U_j
\end{CD}$
\end{center}

\noindent Here $q$ is the quotient map and $\phi_{ji}$ is the map induced by $\widetilde{\phi}_{ji}$. Note the top square always commutes, so the overlapping condition is really about the bottom square. We call $\left|Y^{\text{orb}}\right|$ the \textit{underlying space} of $Y^{\text{orb}}$. We say a 3-orbifold is \textit{oriented} when we have the following: in each chart, $\widetilde{U_i}$ oriented, $G_i$ lies in $SO(3)$, and the action of $G_i$ on $\widetilde{U_i}$ preserves orientation, and on overlaps $U_i \subset U_j$ the embedding $\widetilde{\phi}_{ji}$ preserves orientation. The 3-orbifolds in this paper will be oriented. $Y^{\text{orb}}$ is \textit{connected} (respectively \textit{compact}) when $\left|Y^{\text{orb}}\right|$ is connected (respectively compact).

Given a point $p \in \left|Y^{\text{orb}}\right|$, let $(U, \widetilde{U}, G, \phi)$ be a chart containing $p$ and let $\widetilde{p}$ be a lift of $p$ to $\widetilde{U}$. Then the \textit{local group} $G_p$ is the isotropy group $\{g \in G : g \cdot \widetilde{p} = \widetilde{p}\}$. Note the isomorphism class of $G_p$ does not depend on the choice of chart or lift, so is well-defined. In particular, if we fix a chart but vary the lifts, then the local groups we get are all conjugate. The \textit{singular locus} $\Sigma Y^{\text{orb}}$ of $Y^{\text{orb}}$ is the set of all points $p$ in $|Y^{\text{orb}}|$ with nontrivial local group $G_p$. Note that if the singular locus is empty, then we recover the definition of a 3-manifold. In this paper we will focus on 3-orbifolds with singular locus a knot. By general theory every point on the knot has local group equal to $\mathbb{Z}_n$ for the same $n$. Furthermore, we can identify a neighborhood of the knot with $(D^2  \times S^1) / \mathbb{Z}_n$ where $\mathbb{Z}_n$ acts by rotations about the core circle $0 \times S^1$. Now let $E$ denote the complement of the interior of the neighborhood. Then $H_1^{\text{orb}}(Y^{\text{orb}})$ is defined to be $H_1(E)/ \langle \mu ^n \rangle$, where $\mu$ is a meridian of the singular knot. Note that when $n=1$, $Y^{\text{orb}}$ is just a 3-manifold and $H_1^{\text{orb}}(Y^{\text{orb}})$ is just $H_1(Y^{\text{orb}})$. 

As an example, consider the $n$-fold cyclic branched cover $\Sigma^n(K)$ of $K \subset S^3$. There is a natural action of $\mathbb{Z}_n$ on $\Sigma^n(K)$, and the quotient space $\Sigma^n(K) / \mathbb{Z}_n$ can be thought of as the 3-orbifold $(S^3, K, n)$, where the underlying space is $S^3$, the singular locus is $K$, and every point $y$ on $K$ has isotropy group $G_y$ equal to $\mathbb{Z}_n$. Furthermore, it's not hard to see that $H_1^{\text{orb}}(S^3, K, n) \cong \mathbb{Z}_n$.

Finally, an (orientation-preserving) \textit{homeomorphism} $f: (Y_1, K_1, n) \rightarrow (Y_2, K_2, n)$ between oriented 3-orbifolds $(Y_1, K_1, n)$ and $(Y_2, K_2, n)$ is an (orientation-preserving) homeomorphism $|f|: Y_1 \rightarrow Y_2$ between the underlying oriented 3-manifolds $Y_1$ and $Y_2$ that takes the singular curve $K_1$ to the singular curve $K_2$.   

\subsection{Bordered Heegaard Floer homology}

In this section we give an overview of the bordered Floer invariants. We focus on the torus boundary case because for the most part this is the setting we'll be working in. The details are covered in \cite{LOT1, LOT2, HRW}.

\subsubsection{\textbf{Algebraic preliminaries}}\label{AlgPrelim}

We start by recalling the two algebraic structures (Type D and Type A) that give rise to $\widehat{\mathit{CFD}}$ and $\widehat{\mathit{CFA}}$, the two bordered Floer invariants for the torus boundary case. Let $\mathcal{A}$ be the unital path algebra over $\mathbb{Z}_2$ associated to the quiver in Figure \ref{fig:1} modulo the relations $\rho_2 \rho_1, \rho_3 \rho_2$, in other words we only compose paths when the indices increase. As a $\mathbb{Z}_2$-vector space, $\mathcal{A}$ is generated by eight elements: the two idempotents $\iota_1$ and $\iota_2$, and the six ``Reeb" elements $\rho_1$, $\rho_2$, $\rho_3$, $\rho_{12}:=\rho_1 \rho_2$, $\rho_{23}:=\rho_2 \rho_3$, and $\rho_{123}:=\rho_1 \rho_2 \rho_3$. The multiplicative identity $\mathbf{1}$ in $\mathcal{A}$ is given by $\iota_1 + \iota_2$. We will also need to work with the subalgebra $\mathcal{I}$ generated by  $\iota_1$ and $\iota_2$, this is a commutative ring with multiplicative identity $\mathbf{1} = \iota_1 + \iota_2$.

\begin{figure}[h]
\centering
\begin{tikzcd}
\iota_1 \arrow[r, bend left=50, "\rho_1" ]
   \arrow[r, bend right=50, "\rho_3" ]
   & \iota_2 \arrow[l, "\rho_2" ']
\end{tikzcd}
\caption{Quiver for torus algebra $\mathcal{A}$} \label{fig:1}
\end{figure}
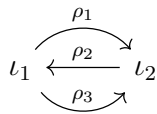

A \textit{(left) type D structure over $\mathcal{A}$} is a pair $\big(N, \delta_1\big)$ consisting of a finite-dimensional $\mathbb{Z}_2$-vector space $N$ that's equipped with a (left) action by $\mathcal{I}$ so that
\begin{equation*}
N = \iota_1N \oplus \iota_2N
\end{equation*}

\noindent as a vector space, together with a map $\delta_1: N \rightarrow \mathcal{A} \underset{\mathcal{I}}{\otimes} N$ that satisfies the following relation
\begin{equation*}
(\mu \otimes id_N) \circ (id_\mathcal{A} \otimes \delta_1) \circ \delta_1 = 0, 
\end{equation*}

\noindent where $\mu: \mathcal{A} \otimes \mathcal{A} \rightarrow \mathcal{A}$ denotes the multiplication in $\mathcal{A}$. Given a type D structure $\big(N, \delta_1\big)$ and $k \in \mathbb{N} \cup \{0\}$, we have maps
\begin{equation*}
\delta_k : N \rightarrow \underbrace{\mathcal{A} \underset{\mathcal{I}}{\otimes} \ldots \underset{\mathcal{I}}{\otimes} \mathcal{A}}_\text{k times} \underset{\mathcal{I}}{\otimes} N
\end{equation*}

\noindent defined inductively as follows: $\delta_0 = id_N$ and $\delta_k = (id_{\mathcal{A}^{\otimes k-1}} \otimes \delta_1) \circ \delta_{k-1}$. We say that $\big(N, \delta_1\big)$ is \textit{bounded} if $\delta_k \equiv 0$ for all $k$ sufficiently large. Note that the above relation on $\delta_1$ can be thought of as $(\mu \otimes id_N) \circ \delta_2 = 0$.

Type D structures $(N, \delta_1)$ can be represented by decorated directed graphs. First choose a basis for $N$ by choosing a basis for each subspace $\iota_{*} N$. Then for each basis element take a vertex. If the basis element lies in $\iota_1N$, decorate the vertex with $\bullet$, otherwise decorate the vertex with $\circ$. Whenever basis elements $x$ and $y$ are related in the following way: $\rho_I \otimes y$ is a summand of $\delta_1(x)$ with $\rho_I \in \{\rho_{\emptyset}:=\textbf{1}, \rho_1, \rho_2, \rho_3, \rho_{12}, \rho_{23}, \rho_{123}\}$, put a directed edge from vertex $x$ to vertex $y$, and decorate the edge with $\rho_I$. The relation on $\delta_1$ then translates into the following condition on the graph: for any directed path of length $2$, the product of the labels equals $0$ in $\mathcal{A}$. The higher maps $\delta_k$ can be recovered by following directed paths of length $k$.

We call a type D structure \textit{reduced} if the associated graph has no edges labelled \textbf{1}. Because of how the idempotents $\iota_1$ and $\iota_2$ interact with the Reeb elements $\rho_1$, $\rho_2$, and $\rho_3$ in $\mathcal{A}$, the graph of any reduced type D structure can only contain edges that look like
\begin{equation*}
\bullet \xrightarrow{\textrm{$\rho_1$}} \circ, \circ \xrightarrow{\textrm{$\rho_2$}} \bullet, \bullet \xrightarrow{\textrm{$\rho_3$}} \circ, \bullet \xrightarrow{\textrm{$\rho_{12}$}} \bullet, \circ \xrightarrow{\textrm{$\rho_{23}$}} \circ, \textrm{or } \bullet \xrightarrow{\textrm{$\rho_{123}$}} \circ.
\end{equation*}

\noindent Conversely, to every directed graph with vertices decorated by $\{\bullet, \circ\}$ and edges of the above form so that for any directed path of length 2 the product of the labels equals $0$ in $\mathcal{A}$, we can associate a (reduced) type D structure $(N, \delta_1)$ as follows. Take $N$ to be the $\mathbb{Z}_2$-vector space generated by the vertices. If we identify $\bullet$ with $\iota_1$ and $\circ$ with $\iota_2$, then we get the following action of $\mathcal{I}$ on $N$: for every vertex $x$ labelled by $\bullet$, set $\iota_1 \cdot x = x$ and $\iota_2 \cdot x = 0$, and for every vertex $x$ labelled by $\circ$, set $\iota_1 \cdot x = 0$ and $\iota_2 \cdot x = x$. The edges encode the map $\delta_1$, and it's clear that $(N, \delta_1)$ forms a reduced type D structure.

A \textit{(right) type A structure over $\mathcal{A}$} is a pair $\big(M, \{m_{k}\}_{k=1}^{\infty}\big)$ consisting of a finite-dimensional $\mathbb{Z}_2$-vector space $M$ that's equipped with a (right) action by $\mathcal{I}$ so that
\begin{equation*}
M = M\iota_1 \oplus M\iota_2
\end{equation*}

\noindent as a vector space, together with multiplication maps
\begin{equation*}
m_{k} : M \underset{\mathcal{I}}{\otimes} \underbrace{\mathcal{A} \underset{\mathcal{I}}{\otimes} \ldots \underset{\mathcal{I}}{\otimes} \mathcal{A}}_\text{k - 1 times} \rightarrow M
\end{equation*}
\noindent that satisfy the following relation for any $x \in M$, $k \in \mathbb{N}$, and $a_1,\ldots , a_{k-1} \in \mathcal{A}$:
\begin{equation*}
\begin{aligned}[t]
0 &= \sum\limits_{j=1}^k m_{k-j+1}\big(m_j(x \otimes a_1 \ldots \otimes a_{j-1}) \otimes a_j \otimes \ldots \otimes a_{k-1}\big) \\ & + \sum\limits_{j=1}^{k-2} m_{k-1}(x \otimes a_1 \ldots \otimes a_{j-1} \otimes a_j a_{j+1} \otimes a_{j+2} \otimes \ldots \otimes a_{k-1}).
\end{aligned}
\end{equation*}

\noindent A type A structure $\big(M, \{m_{k}\}_{k=1}^{\infty}\big)$ is said to be 
\begin{enumerate}

\item $\textit{unital}$ if

\begin{itemize}
\item $\begin{aligned}[t]
m_2(x, \mathbf{1}) = x, \textrm{and}
\end{aligned}$

\item $\begin{aligned}[t]
m_k(x, a_1, \dots , a_{k-1}) = 0, \textrm{ for } k \geq 3 \textrm{ and at least one } a_i = \mathbf{1}, \textrm{and}
\end{aligned}$
\end{itemize}

\item \textit{bounded} if $m_k \equiv 0$ for all $k$ sufficiently large.

\end{enumerate}

Using an algorithm by Hedden and Levine \cite[Theorem 2.2]{HL}, one can construct a (non-unital) type A structure $\big(M, \{m_{k}\}_{k=1}^{\infty}\big)$ from a (reduced) type D structure $(N, \delta_1)$. We keep $M$ the same as $N$, both in terms of underlying vector space and idempotent action, and dualize the map $\delta_1$ to maps $m_{k}$ by doing the following. First relabel the edges of the graph that's associated to $(N, \delta_1)$ by swapping indices $1$ and $3$, keeping index $2$ the same. Next represent every directed path in the new graph by a string of numbers, by concatenating the indices. For example, the directed path $\bullet\xrightarrow{\textrm{$\rho_1$}} \circ \xrightarrow{\textrm{$\rho_{21}$}} \circ$ gives the string $121$. Then rewrite every string of numbers as a string of increasing sequences $I=I_1, \ldots, I_{k-1}$ so that the last element of $I_j$ is bigger than the first element of $I_{j+1}$. For example, the string $121$ gets rewritten as $12, 1$. For every directed path with source vertex $x$, target vertex $y$, and associated string $I=I_1, \ldots, I_{k-1}$, we define $m_k(x \otimes \rho_{I_1} \otimes \ldots \otimes \rho_{I_{k-1}}) = y$. For everything else, we define the multiplication to be zero. As an example, consider the type D directed path $\overset{x}{\bullet} \xrightarrow{\textrm{$\rho_3$}} \circ \xrightarrow{\textrm{$\rho_{23}$}} \overset{y}{\circ}$. It gives rise to the multiplication $m_3(x, \rho_{12}, \rho_1)=y$. 

If $\big(M, \{m_{k}\}_{k=1}^{\infty}\big)$ is a type A structure over $\mathcal{A}$, $\big(N, \delta_1\big)$ is a type D structure over $\mathcal{A}$, and at least one of them is bounded, then we can form the \emph{box tensor product} $M \boxtimes N$, a $\mathbb{Z}_2$-chain complex $(M \underset{\mathcal{I}}{\otimes} N, \delta^{\boxtimes})$ with differential $\delta^{\boxtimes} : M \underset{\mathcal{I}}{\otimes} N \rightarrow M \underset{\mathcal{I}}{\otimes} N$ given by
\begin{equation*}
\delta^{\boxtimes}(x \otimes y) = \sum_{k=0}^{\infty} \big(m_{k+1} \otimes id_N\big)\big(x \otimes \delta_k(y)\big).
\end{equation*}

In addition to type D and type A structures over $\mathcal{A}$, we will also need to work with with \textit{type DA structures over $(\mathcal{A}, \mathcal{A})$}. This is a $\mathbb{Z}_2$-vector space $N$ with the structure of an $(\mathcal{I}, \mathcal{I})$-bimodule, together with maps 
\begin{equation*}
\delta^{k}_1 :N \underset{\mathcal{I}}{\otimes} \underbrace{\mathcal{A} \underset{\mathcal{I}}{\otimes} \ldots \underset{\mathcal{I}}{\otimes} \mathcal{A}}_\text{k - 1 times} \rightarrow \mathcal{A} \underset{\mathcal{I}}{\otimes} N
\end{equation*}

\noindent that satisfy a compatibility condition similar to the one for type D structures (see \cite{LOT2}, Definition 2.2.43). Similar to type A structures, a type DA structure $\big(N, \{\delta_{1}^k\}_{k=1}^{\infty}\big)$ is \textit{unital} if we have the following:

\begin{enumerate}
\item $\begin{aligned}[t]
\delta^{2}_1(x, \mathbf{1}) = \mathbf{1} \otimes x, \textrm{ and}
\end{aligned}$

\item $\begin{aligned}[t]
\delta^{k}_1(x, a_1, \dots , a_{k-1}) = 0, \textrm{ when } k \geq 3 \textrm{ and at least one } a_i = \mathbf{1}.
\end{aligned}$
\end{enumerate}

\noindent All of our type DA structures will be unital. Like with type D and type A structures, we can take the \textit{box tensor product} of a type DA structure with a type D structure, or the \textit{box tensor product} of a type A structure with a type DA structure, when at least one of the factors is bounded. For details, see \cite[Definition 2.3.9]{LOT2}.

\subsubsection{\textbf{Invariants for bordered 3-manifolds}}

A \textit{bordered 3-manifold} is a pair $(Y, \phi)$ consisting of a connected, compact, oriented 3-manifold $Y$ with connected boundary, together with a homeomorphism $\phi$ from a fixed model surface $F$ to the boundary of $Y$. Two bordered 3-manifolds $(Y_1, \phi_1)$ and $(Y_2, \phi_2)$ are called \textit{equivalent} if there is an orientation-preserving homeomorphism $\psi : Y_1 \rightarrow Y_2$ so that $\phi_2 = \psi |_{\partial} \circ \phi_1$.

As noted earlier, we will restrict to the case of torus boundary. Then $F$ is the oriented torus associated to the pointed matched circle $\mathcal{Z}$ in Figure \ref{fig:PointedMatchedCircle}, with 1-handles represented by $\alpha_1^{a}$ and $\alpha_2^{a}$, and orientation given by $\langle \alpha_1^{a}, \alpha_2^{a} \rangle$. If $\phi$ is orientation-preserving, $(Y, \phi)$ is said to be \textit{type A}, otherwise $(Y, \phi)$ is said to be \textit{type D}.

\begin{figure}[h]
\centering{
\resizebox{40mm}{!}{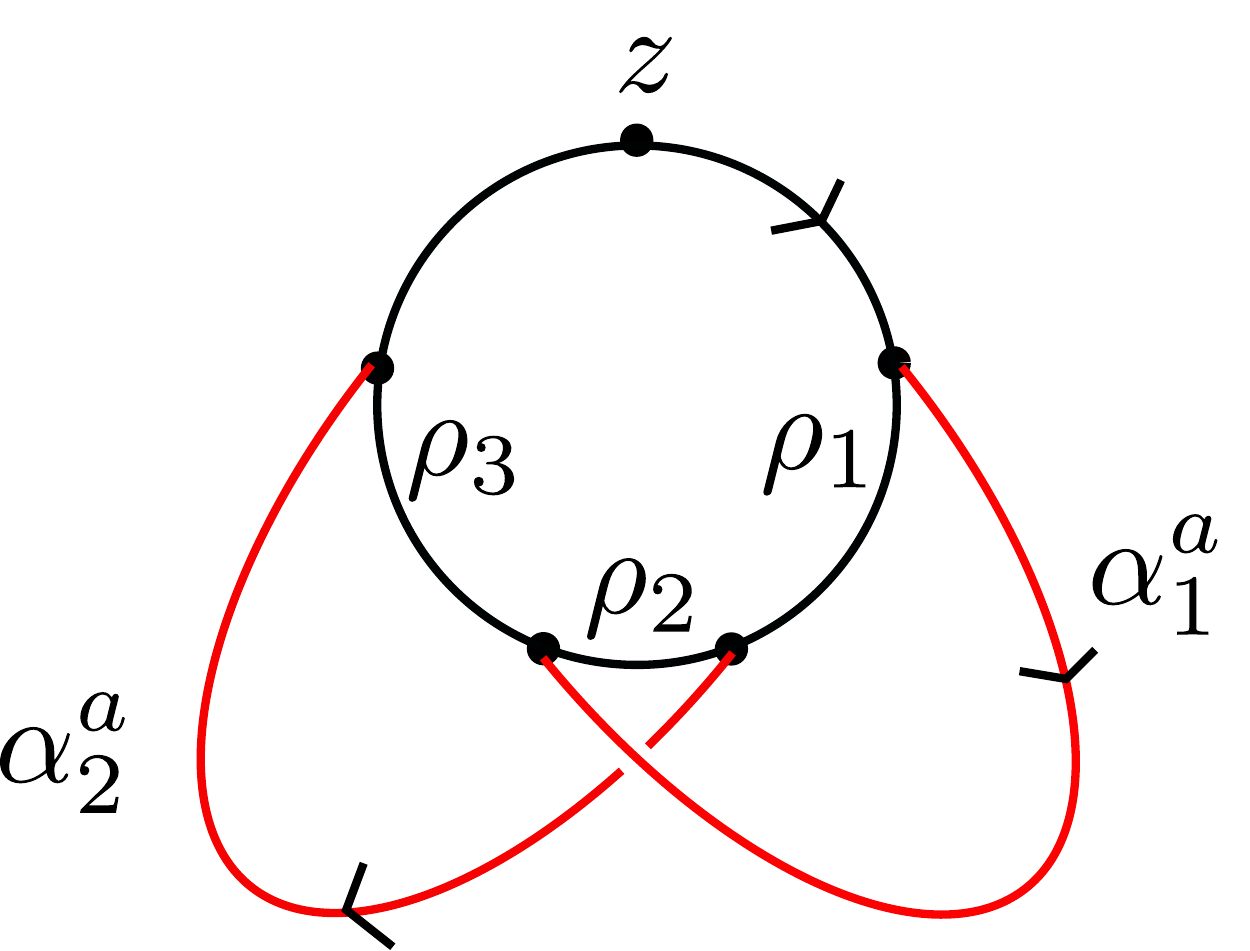}
\caption{Pointed matched circle $\mathcal{Z}$ for torus $F$}
\label{fig:PointedMatchedCircle}
}
\end{figure}

Any bordered 3-manifold $(Y, \phi)$ can be represented by a (sufficiently admissible) \textit{bordered Heegaard diagram} $\mathcal{H}$. This is a tuple 
\begin{equation*}
(\overline{\Sigma}; \underbrace{\underbrace{ \{ \alpha_1^c, \ldots, \alpha^c_{g-1} \} }_{\bm{\alpha}^c}; \underbrace{ \{ \alpha^{a}_1, \alpha^{a}_2 \}}_{\bm{\alpha}^a } }_{\bm{\alpha}}; \underbrace{ \{ \beta_1, \ldots, \beta_g\}}_{\bm{\beta} }; z)
\end{equation*}

\noindent consisting of 

\begin{itemize}
\item a connected, compact, oriented surface $\overline{\Sigma}$ of genus $g$ with connected boundary,
\item two sets $\bm{\alpha}^c$ and $\bm{\beta}$ of pairwise disjoint circles in the interior of $\overline{\Sigma}$,
\item pairwise disjoint properly embedded arcs $\alpha^{a}_1$ and $\alpha^{a}_2$ in $\overline{\Sigma}$, and
\item a point $z$ on $\partial \overline{\Sigma}$ missing the endpoints of $\alpha^{a}_1$ and $\alpha^{a}_2$
\end{itemize}

\noindent so that $\bm{\alpha}^c$ and $\bm{\alpha}^a$ are disjoint, and $ \overline{\Sigma} - \bm{\alpha} $ and $ \overline{\Sigma} - \bm{\beta} $ are connected. To recover $Y$, we attach 2-handles to $\overline{\Sigma} \times I$ along the $\bm{\alpha}^c$ circles in $\overline{\Sigma} \times \{0\}$ and the $\bm{\beta}$ circles in $\overline{\Sigma} \times \{1\}$. The parameterization $\phi$ of $\partial Y$ is specified by the pointed matched circle $(\partial \overline{\Sigma}, \alpha^{a}_1, \alpha^{a}_2, z)$ coming from $\mathcal{H}$, where $\partial \overline{\Sigma}$ is given the induced boundary orientation. If $(\partial \overline{\Sigma}, \alpha^{a}_1, \alpha^{a}_2, z)$ is identified with $\mathcal{Z}$, then $\phi$ is orientation-preserving, and $\mathcal{H}$ describes a type A bordered 3-manifold $(Y, \phi)$, otherwise we're identifying $(\partial \overline{\Sigma}, \alpha^{a}_1, \alpha^{a}_2, z)$ with $-\mathcal{Z}$, and we get a type D bordered 3-manifold $(Y, \phi)$. See Figure \ref{fig:BHD} for an example of a type D bordered 3-manifold.

\begin{figure}[h]
    \centering
    \begin{subfigure}[t]{0.5\textwidth}
        \centering
        \includegraphics[height=1.5in]{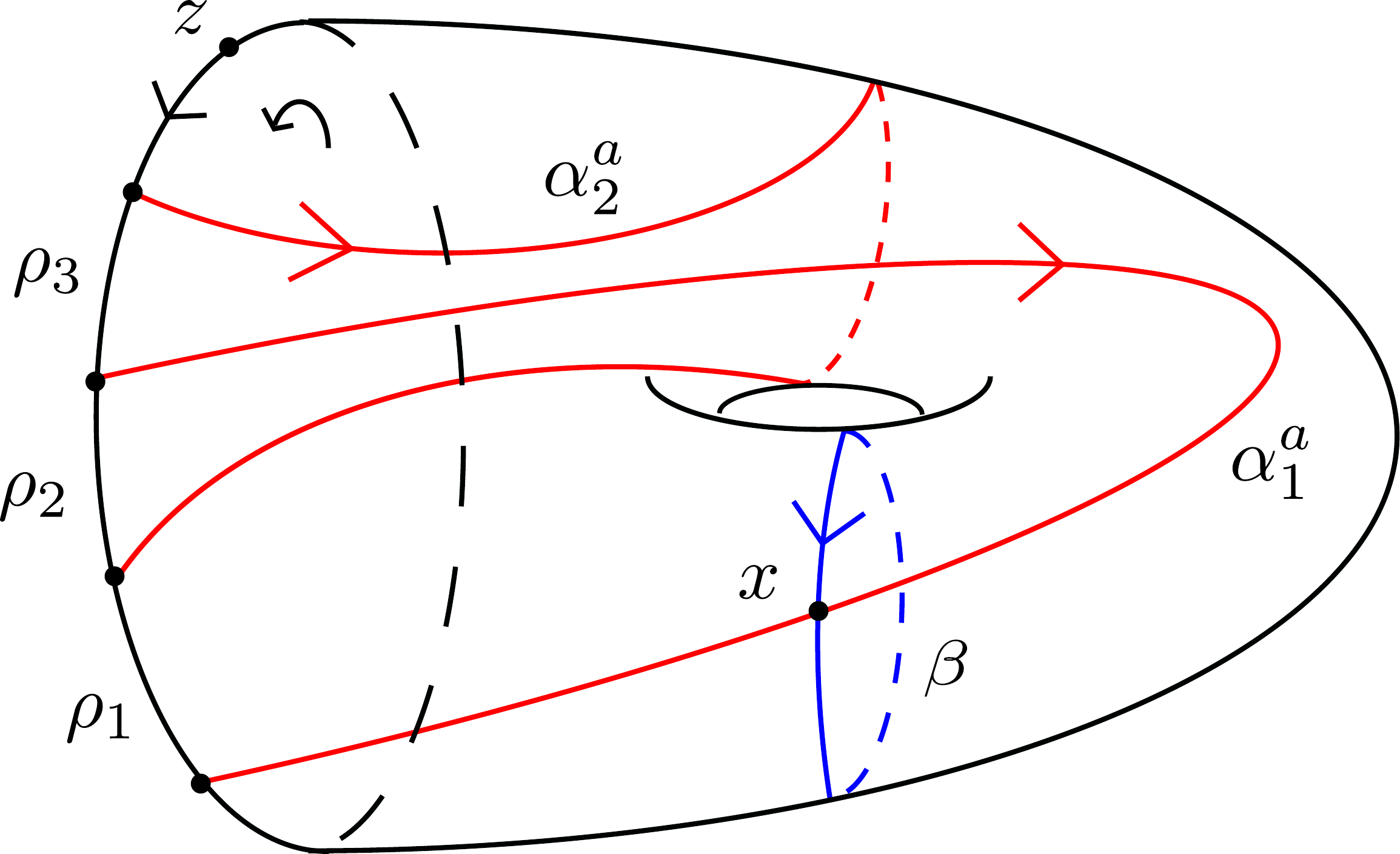}
        \caption{}
    \end{subfigure}%
    ~ 
    \begin{subfigure}[t]{0.5\textwidth}
        \centering
        \includegraphics[height=1.5in]{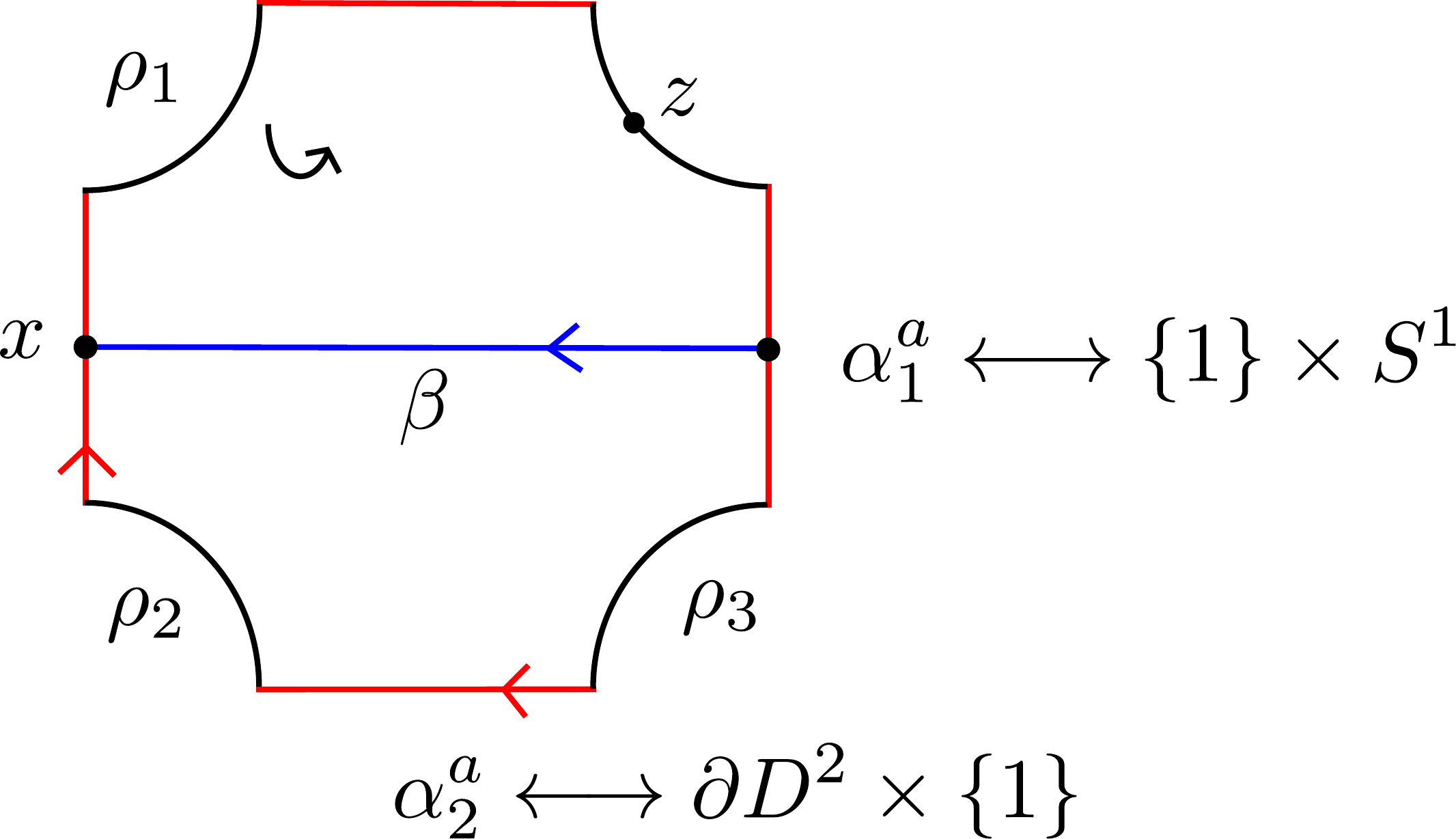}
        \caption{}
    \end{subfigure}
    \caption{On the left, a genus 1 bordered Heegaard diagram for $D^2 \times S^1$ with standard product orientation and boundary parameterized by $\alpha_1^a \mapsto \{1\} \times S^1$ and $\alpha_2^a \mapsto \partial D^2 \times \{1\}$. On the right, the same bordered Heegaard diagram thought of as a decorated square missing four corners with opposite sides identified.}
    \label{fig:BHD}
\end{figure}

Bordered Floer theory, as defined by Lipshitz, Ozsv\'{a}th, and D. Thurston in \cite{LOT1, LOT2}, associates to a bordered Heegaard diagram $\mathcal{H}$ representing a bordered 3-manifold $(Y, \phi)$ a type A structure $(\widehat{\mathit{CFA}}(\mathcal{H}), \{m_k\}_{k=1}^{\infty})$ if $(Y, \phi)$ is type A, and a type D structure $(\widehat{\mathit{CFD}}(\mathcal{H}), \delta_1)$ if $(Y, \phi)$ is type D. As $\mathbb{Z}_2$-vector spaces, $\widehat{\mathit{CFA}}(\mathcal{H})$ and $\widehat{\mathit{CFD}}(\mathcal{H})$ are generated by $g$-tuples $\bm{x}$ of points in $\bm{\alpha} \cap \bm{\beta}$ with one point on each $\bm{\alpha}^c$ circle, one point on each $\bm{\beta}$ circle, and one point on one of the $\bm{\alpha}^a$ arcs. The right $\mathcal{I}$-action on $\widehat{\mathit{CFA}}(\mathcal{H})$ is given by
\[
\bm{x} \cdot \iota_1 = \begin{cases}
    \bm{x}, & \text{if \textit{\textbf{x}} occupies the $\alpha^a_1$ arc} \\
    0,              & \text{otherwise}\end{cases}    
\]

\[
\bm{x} \cdot \iota_2 = \begin{cases}
    \bm{x}, & \text{if \textit{\textbf{x}} occupies the $\alpha^a_2$ arc} \\
    0,              & \text{otherwise}\end{cases}
\]

\noindent while the left $\mathcal{I}$-action on $\widehat{\mathit{CFD}}(\mathcal{H})$ is given by
\[
\iota_1 \cdot \bm{x} = \begin{cases}
    \bm{x}, & \text{if \textit{\textbf{x}} occupies the $\alpha^a_2$ arc} \\
    0,              & \text{otherwise}\end{cases}    
\]

\[
\iota_2 \cdot \bm{x} = \begin{cases}
    \bm{x}, & \text{if \textit{\textbf{x}} occupies the $\alpha^a_1$ arc} \\
    0,              & \text{otherwise.}\end{cases}
\]
   
\noindent The type A and type D structure maps
\[
m_{k} : \widehat{\mathit{CFA}}(\mathcal{H}) \underset{\mathcal{I}}{\otimes} \mathcal{A} \underset{\mathcal{I}}{\otimes} \ldots \underset{\mathcal{I}}{\otimes} \mathcal{A} \rightarrow \widehat{\mathit{CFA}}(\mathcal{H})
\]

\noindent and
\[
\delta_1: \widehat{\mathit{CFD}}(\mathcal{H}) \rightarrow \mathcal{A} \underset{\mathcal{I}}{\otimes} \widehat{\mathit{CFD}}(\mathcal{H}) 
 \]
 
\noindent are defined by counting certain $J$-holomorphic curves in $\Sigma \times [0,1] \times \mathbb{R}$, for a sufficiently nice almost complex structure $J$ on $\Sigma \times [0,1] \times \mathbb{R}$, with $\Sigma$ the interior of $\overline{\Sigma}$. Details can be found in \cite[Chapters 6 and 7]{LOT1}. Up to homotopy equivalence, the type A and type D structures $(\widehat{\mathit{CFA}}(\mathcal{H}), \{m_k\}_{k=1}^{\infty})$ and $(\widehat{\mathit{CFD}}(\mathcal{H}),\delta_1)$ don't depend on the choice of $J$, and so we get invariants of $\mathcal{H}$. Because different bordered Heegaard diagrams for equivalent bordered 3-manifolds produce homotopy equivalent bordered invariants, this process gives us an invariant of any bordered 3-manifold $(Y, \phi)$ considered up to equivalence. If $(Y, \phi)$ is of type A, we denote the invariant by $\widehat{\mathit{CFA}}(Y, \phi)$, and if $(Y, \phi)$ is of type D, we denote the invariant by $\widehat{\mathit{CFD}}(Y, \phi)$.

As an example, consider $D^2 \times S^1$ with boundary parameterization $\psi: F \rightarrow \partial (D^2 \times S^1)$ defined by $\alpha_1^a \mapsto \{1\}\times S^1$ and $\alpha_2^a \mapsto  \partial D^2 \times \{1\} $. Using the bordered Heegaard diagram in Figure \ref{fig:BHD} for $(D^2 \times S^1, \psi)$, we get that $\widehat{\mathit{CFD}}(D^2 \times S^1, \psi)$ is given by the decorated, directed graph in Figure \ref{fig:CFD_Solid_Torus}.

\begin{figure}[h]
\centering{
\resizebox{24mm}{!}{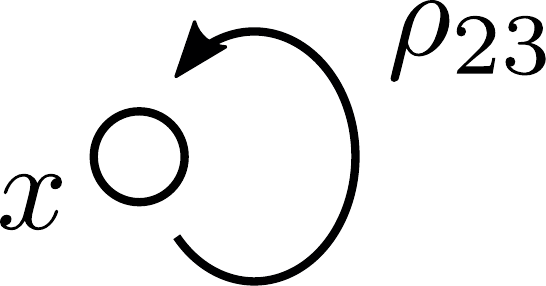}
\caption{The type D structure for $(D^2 \times S^1, \psi)$}
\label{fig:CFD_Solid_Torus}
}
\end{figure}

When we vary the parameterization of the boundary of $(Y, \phi)$, the bordered invariants $\widehat{\mathit{CFA}}(Y, \phi)$ and $\widehat{\mathit{CFD}}(Y, \phi)$ change by a type DA structure over $(\mathcal{A}, \mathcal{A})$. Specifically, given an orientation-preserving homeomorphism $\psi$ of the model torus $F$, there exists a type DA structure $\widehat{\mathit{CFDA}}(\psi)$ so that
\[
\widehat{\mathit{CFDA}}(\psi) \boxtimes \widehat{\mathit{CFD}}(Y, \phi) \simeq \widehat{\mathit{CFD}}(Y, \phi \circ \psi)
\]
\noindent as type D structures over $\mathcal{A}$, and
\begin{equation}\label{changeofparameqn}
\widehat{\mathit{CFA}}(Y, \phi) \boxtimes \widehat{\mathit{CFDA}}(\psi) \simeq \widehat{\mathit{CFA}}(Y, \phi \circ \psi^{-1})
\end{equation}
\noindent as type A structures over $\mathcal{A}$. For details, see \cite[Theorem 2]{LOT2}.
    
Given a type A bordered 3-manifold $(Y_1, \phi_1)$ and a type D bordered 3-manifold $(Y_2, \phi_2)$, we can build a closed, oriented, smooth 3-manifold $Y$ by gluing $Y_1$ and $Y_2$ together along their boundaries via the homeomorphism $\phi_2 \circ \phi_1^{-1} : \partial Y_1 \rightarrow \partial Y_2$. To the bordered pieces $(Y_1, \phi_1)$ and $(Y_2, \phi_2)$ we associate the bordered invariants $\widehat{\mathit{CFA}}(Y_1, \phi_1)$ and $\widehat{\mathit{CFD}}(Y_2, \phi_2)$, and to $Y$ we associate the hat flavor of Heegaard Floer homology $\widehat{\mathit{HF}}(Y)$. The \textit{pairing theorem} tells us that if at least one of the bordered invariants is bounded, then $\widehat{\mathit{HF}}(Y)$ is determined by $\widehat{\mathit{CFA}}(Y_1, \phi_1)$ and $\widehat{\mathit{CFD}}(Y_2, \phi_2)$: 
\begin{equation}\label{pairing}
\widehat{\mathit{HF}}(Y) \cong H_{*}\big( \widehat{\mathit{CFA}}(Y_1) \boxtimes \widehat{\mathit{CFD}}(Y_2)\big).
\end{equation}

\noindent This will motivate our definition of the orbifold Heegaard Floer invariant.

\subsection{$\widehat{\mathit{CFA}}$ of bordered knot exteriors}\label{knotexteriors}

Let $Y$ be the exterior of a knot $K \subset S^3$. Given $r \in \mathbb{Z}$, let $\phi_r: F \rightarrow \partial Y$ to be an orientation-preserving parameterization that sends $\alpha_1^a$ to a meridian $m$ of $K$ and $\alpha_2^a$ to an $r$-framed longitude $\gamma$ of $K$. In this section we recall the algorithm for computing $\widehat{\mathit{CFA}}(Y, \phi_r)$ from the knot Floer chain complex $\mathit{CFK}^-(K)$. This is due to Lipshitz, Ozsv\'{a}th, and D. Thurston in \cite[Theorems 11.26 and A.11]{LOT1} \big(technically their algorithm computes $\widehat{\mathit{CFD}}(Y, \alpha_1^a \mapsto \gamma, \alpha_2^a \mapsto m)$, but by \cite[Theorem 2.2]{HL} we can pass from $\widehat{\mathit{CFD}}(Y, \alpha_1^a \mapsto \gamma, \alpha_2^a \mapsto m)$ to $\widehat{\mathit{CFA}}(Y, \phi_r)$\big).

We start by recalling the definition of $\mathit{CFK}^-(K)$. The details can be found in \cite{OS8, Ras}. First take a doubly-pointed Heegaard diagram $(\Sigma, \bm{\alpha}, \bm{\beta}, w, z)$ of genus $g$ for $K \subset S^3$. If we ignore the base point $z$, then we get a pointed Heegaard diagram $(\Sigma, \bm{\alpha}, \bm{\beta}, w)$ of genus $g$ for $S^3$. To this we can associate the $\mathbb{Z}_2 [U]$-chain complex $(\mathit{CF}^-(S^3), \partial^-)$, where 

\begin{itemize}
\item $\mathit{CF}^-(S^3)$ is the finite-dimensional $\mathbb{Z}_2 [U]$-vector space generated by $g$-tuples $\bm{x}$ of points in $\bm{\alpha} \cap \bm{\beta}$ with one point on each $\bm{\alpha}$ circle and one point on each $\bm{\beta}$ circle, and
\item the diffferential $\partial^-: \mathit{CF}^-(S^3) \rightarrow \mathit{CF}^-(S^3)$ is given by counting certain pseudo-holomorphic curves in $\mathrm{Sym}^g (\Sigma)$.
\end{itemize}

\noindent When we bring back the $z$ base point, which we should think of as representing the knot $K$, we get a $\mathbb{Z}$-grading on $\mathit{CF}^-(S^3)$, called the \textit{Alexander grading}. This is a function $A: \mathit{CF}^-(S^3) \rightarrow \mathbb{Z}$ that satisfies the property $A(U^i \bm{x})=A(\bm{x})-i$. Using $A$, we can define a $\mathbb{Z}$-filtration $\{\mathcal{F}_i\}$ on the $\mathbb{Z}_2 [U]$-chain complex $(\mathit{CF}^-(S^3), \partial^-)$, where each $\mathcal{F}_i$ is a $\mathbb{Z}_2 [U]$-module and $\partial^{-}(\mathcal{F}_i) \subseteq \mathcal{F}_i$. Then $\mathit{CFK}^-(K)$ is defined to be the $\mathbb{Z}_2 [U]$-chain complex $(\mathit{CF}^-(S^3), \partial^-)$ with this $\mathbb{Z}$-filtration $\{\mathcal{F}_i\}$.

By negating the powers of $U$, we get a second $\mathbb{Z}$-filtration $I$ on $\mathit{CFK}^-(K)$. We can visualize $\mathit{CFK}^-(K)$, together with the $I$ filtration, as a directed graph in $\mathbb{Z} \times \mathbb{Z} \subset \mathbb{R} \times \mathbb{R}$ as follows. First pick a basis $\{\bm{x_0}, \ldots, \bm{x_{2n}}\}$ for $\mathit{CFK}^-(K)$ over $\mathbb{Z}_2 [U]$ as above. Then $\{U^i \bm{x_k}\ | \ i \in \mathbb{Z} \textrm{ and } k=0,  \ldots, 2n\}$ is a basis for $\mathit{CFK}^-(K)$ over $\mathbb{Z}_2$, and it's these elements that form the vertices of our graph, with $U^i \bm{x_k}$ at point $\big(I(U^i \bm{x_k}), A(U^i \bm{x_k})\big)= \big(-i, A(\bm{x_k})-i\big)$ in $\mathbb{Z} \times \mathbb{Z} \subset \mathbb{R} \times \mathbb{R}$. The edges of the graph are given by the differential $\partial^-$, namely we draw a directed edge from $U^i \bm{x_k}$ to $U^j \bm{x_l}$ if $\partial^- (U^i \bm{x_k})$ contains $U^j \bm{x_l}$ as a summand. Note that the graph of $\mathit{CFK}^-(K)$ lies in the part of the $(I,A)$-plane with $I \leq 0$. 

Let $C^{\textrm{vert}}$ be the $\mathbb{Z}_2$-chain complex $\mathit{CFK}^-(K) / \big(U \cdot \mathit{CFK}^-(K)\big)$. We'll denote the differential by $\partial^{\textrm{vert}}$, and call $C^{\textrm{vert}}$ as the \textit{vertical complex} associated to $\mathit{CFK}^-(K)$. If we think of $\mathit{CFK}^-(K)$ as a directed graph in $\mathbb{Z} \times \mathbb{Z} \subset \mathbb{R} \times \mathbb{R}$, then the graph of $C^{\textrm{vert}}$ is the part of $\mathit{CFK}^-(K)$ that lies on the vertical $A$-axis (with directed edges pointing down).

To $\mathit{CFK}^-(K)$ with the Alexander filtration $\{\mathcal{F}_i\}$, we can associate the finitely generated, free $\mathbb{Z}_2[U]$-module 
\[
\textrm{gr}(\mathit{CFK}^-(K)) := \bigoplus\limits_{i \in \mathbb{Z}} \mathcal{F}_i / \mathcal{F}_{i-1}.
\]
\noindent Given any $x \in \mathit{CFK}^-(K)$, denote by $[x]$ the image of $x$ in $\mathcal{F}_{A(x)} / \mathcal{F}_{A(x)-1}$. We call a basis $\{\bm{x_0}, \ldots, \bm{x_{2n}}\}$ for $\mathit{CFK}^-(K)$ over $\mathbb{Z}_2[U]$ \textit{filtered} if $\{[\bm{x_0}], \ldots, [\bm{x_{2n}}]\}$ is a basis for $\textrm{gr}(\mathit{CFK}^-(K))$. We will be interested in filtered bases for $\mathit{CFK}^-(K)$ that take a particularly simple form, which we now describe.

Let $\mathit{CFK}^{\infty}(K)$ denote the $\mathbb{Z}_2[U, U^{-1}]$-chain complex $\mathit{CFK}^-(K) \otimes_{\mathbb{Z}_2[U]} \mathbb{Z}_2[U, U^{-1}]$. There is a natural way to extend the Alexander and $I$ filtrations on $\mathit{CFK}^-(K)$ to $\mathit{CFK}^{\infty}(K)$. Then we can view $\mathit{CFK}^{\infty}(K)$ as a directed graph in $\mathbb{Z} \times \mathbb{Z} \subset \mathbb{R} \times \mathbb{R}$, with $\mathit{CFK}^-(K)$ as a subgraph. To $\mathit{CFK}^{\infty}(K)$ we can associate the $\mathbb{Z}_2$-chain complex
\[
C^{\textrm{horz}} := \mathcal{F}_0 \big(\mathit{CFK}^{\infty}(K)\big) / \mathcal{F}_{-1} \big(\mathit{CFK}^{\infty}(K)\big)
\]
with differential denoted by $\partial^{\textrm{horz}}$. We'll refer to this as the \textit{horizontal complex} associated to $\mathit{CFK}^{\infty}(K)$. If we view $\mathit{CFK}^{\infty}(K)$ as a directed graph in $\mathbb{Z} \times \mathbb{Z} \subset \mathbb{R} \times \mathbb{R}$, then $C^{\textrm{horz}}$ can be thought of as the part of $\mathit{CFK}^{\infty}(K)$ lying on the horizontal $I$-axis (with directed edges pointing to the left).

We're now ready to define those nice filtered bases for $\mathit{CFK}^-(K)$. Let $\{\bm{x_0}, \ldots, \bm{x_{2n}}\}$ be a filtered basis for $\mathit{CFK}^-(K)$, and let $\{\overline{\bm{x_0}}, \ldots, \overline{\bm{x_{2n}}}\}$ denote the induced basis for the vertical complex $C^{\textrm{vert}}$. We define $\{\bm{x_0}, \ldots, \bm{x_{2n}}\}$ to be \textit{vertically simplified} if each basis element $\overline{\bm{x_i}}$ satisfies one of the following:

\begin{itemize}
\item $\overline{\bm{x_i}} \in \textrm{im}(\partial ^{\textrm{vert}}) \subseteq \textrm{ker}(\partial^{\textrm{vert}})$ and $\partial ^{\textrm{vert}}(\overline{\bm{x_{i-1}}})=\overline{\bm{x_i}}$,
\item $\overline{\bm{x_i}} \in \textrm{ker}(\partial^{\textrm{vert}})$, but $\overline{\bm{x_i}} \notin \textrm{im}(\partial ^{\textrm{vert}})$, or
\item $\overline{\bm{x_i}} \notin\textrm{ker}(\partial^{\textrm{vert}})$ and $\partial ^{\textrm{vert}}(\overline{\bm{x_{i}}})=\overline{\bm{x_{i+1}}}$.
\end{itemize}

\noindent When $\partial ^{\textrm{vert}}(\overline{\bm{x_{i}}})=\overline{\bm{x_{i+1}}}$, we say that there is a \textit{vertical arrow} from $\bm{x_i}$ to $\bm{x_{i+1}}$ of length $A(\bm{x_i}) - A(\bm{x_{i+1}})$. Because $H_{*}(C^{\textrm{vert}}) \cong \mathbb{Z}_2$ and $\partial ^{\textrm{vert}}$ pairs up basis elements in $\{\overline{\bm{x_0}}, \ldots, \overline{\bm{x_{2n}}}\}$, there is a distinguished basis element in $\{\overline{\bm{x_0}}, \ldots, \overline{\bm{x_{2n}}}\}$ with no incoming and outgoing vertical arrows. Without loss of generality, we assume it's $\overline{\bm{x_0}}$, and we call $\bm{x_0}$ the $\textit{generator}$ of the vertical complex $C^{\textrm{vert}}$.

There is a horizontal analogue of the above definition. Given a filtered basis $\{\bm{y_0}, \ldots, \bm{y_{2n}}\}$ for $\mathit{CFK}^-(K)$, we can define a basis $\{\overline{U^{A(\bm{y_0})} \bm{y_0}}, \ldots, \overline{U^{A(\bm{y_{2n}})} \bm{y_{2n}}}\}$ for $C^{\textrm{horz}}$. Then $\{\bm{y_0}, \ldots, \bm{y_{2n}}\}$ is called \textit{horizontally simplified} if each basis element $\overline{U^{A(\bm{y_i})} \bm{y_i}}$ satisfies one of the following:

\begin{itemize}
\item $\overline{U^{A(\bm{y_i})} \bm{y_i}} \in \textrm{im}(\partial ^{\textrm{horz}}) \subseteq \textrm{ker}(\partial^{\textrm{horz}})$ and $\partial ^{\textrm{horz}}\big(\overline{U^{A(\bm{y_{i-1}})} \bm{y_{i-1}}}\big)=\overline{U^{\bm{A(y_i)}} \bm{y_i}}$,
\item $\overline{U^{A(\bm{y_i})} \bm{y_i}}\in \textrm{ker}(\partial^{\textrm{horz}})$, but $\overline{U^{\bm{A(y_i)}} \bm{y_i}} \notin \textrm{im}(\partial ^{\textrm{horz}})$, or
\item $\overline{U^{A(\bm{y_i})} \bm{y_i}} \notin \textrm{ker}(\partial^{\textrm{horz}})$ and $\partial ^{\textrm{horz}}\big(\overline{U^{\bm{A(y_i)}} \bm{y_i}}\big)=\overline{U^{A(\bm{y_{i+1}})} \bm{y_{i+1}}}$.
\end{itemize}

\noindent When $\partial ^{\textrm{horz}}\big(\overline{U^{A(\bm{y_i})} \bm{y_i}}\big)=\overline{U^{A(\bm{y_{i+1}})} \bm{y_{i+1}}}$, we say that there is a \textit{horizontal arrow} from $\overline{U^{A(\bm{y_i})} \bm{y_i}}$ to $\overline{U^{A(\bm{y_{i+1}})} \bm{y_{i+1}}}$ of length $A(\bm{y_{i+1}}) - A(\bm{y_{i}})$. Like in the vertical case, $H_{*}(C^{\textrm{horz}}) \cong \mathbb{Z}_2$ and $\partial ^{\textrm{horz}}$ pairs up basis elements in $\{\overline{U^{A(\bm{y_0})} \bm{y_0}}, \ldots, \overline{U^{A(\bm{y_{2n}})} \bm{y_{2n}}}\}$, so there is a distinguished basis element in $\{\overline{U^{A(\bm{y_0})} \bm{y_0}}, \ldots, \overline{U^{A(\bm{y_{2n}})} \bm{y_{2n}}}\}$ with no incoming and outgoing horizontal arrows. Without loss of generality, we assume it's $\overline{U^{A(\bm{y_0})} \bm{y_0}}$, and we call $\bm{y_0}$ the $\textit{generator}$ of the horizontal complex $C^{\textrm{horz}}$.

We can now explain how to go from $\mathit{CFK}^-(K)$ to a decorated, directed graph that describes $\widehat{\mathit{CFA}}(Y, \phi_r)$. First, take a vertically simplified basis $\{\bm{w_i}\}$ for $\mathit{CFK}^-(K)$. Since we can identify the vertical complex $C^{\textrm{vert}}$ with $\widehat{\mathit{CFA}}(Y, \phi_r) \cdot \iota_1$, $\{\bm{w_i}\}$ (or really $\{\overline{\bm{w_i}}\}$) induces a basis for $\widehat{\mathit{CFA}}(Y, \phi_r) \cdot \iota_1$. We represent each of these basis elements in $\widehat{\mathit{CFA}}(Y, \phi_r) \cdot \iota_1$ by a $\bullet$-labelled vertex. Next, for each vertical arrow from $\bm{w_i}$ to $\bm{w_{i+1}}$ of length $\ell_i$, we introduce $\ell_i$ basis elements $\bm{\kappa_1^i}, \ldots, \bm{\kappa_{\ell_i}}^i$ for $\widehat{\mathit{CFA}}(Y, \phi_r) \cdot \iota_2$ (thought of as vertices labelled by $\circ$) and differentials
\[
\stackrel{\bm{w_i}}{\bullet} \xrightarrow{\textrm{$\rho_3$}} \stackrel{\bm{\kappa_1^i}}{\circ} \xleftarrow{\textrm{$\rho_{21}$}} \ldots \xleftarrow{\textrm{$\rho_{21}$}} \stackrel{\bm{\kappa_{\ell_i}^i}}{\circ} \xleftarrow{\textrm{$\rho_{321}$}} \stackrel{\bm{w_{i+1}}}{\bullet}.    
\]
\indent Now take a horizontally simplified basis $\{\bm{w_i'}\}$ for $\mathit{CFK}^-(K)$. In a similar way, we can identify the horizontal complex $C^{\textrm{horz}}$ with $\widehat{\mathit{CFA}}(Y, \phi_r) \cdot \iota_1$, and so $\{\bm{w_i'}\}$ induces a basis for $\widehat{\mathit{CFA}}(Y, \phi_r) \cdot \iota_1$. We'll think of each of these basis element in $\widehat{\mathit{CFA}}(Y, \phi_r) \cdot \iota_1$ as a vertex labelled by $\bullet$. For each horizontal arrow from $\bm{w_i'}$ to $\bm{w_{i+1}'}$ of length $\ell_i'$, we introduce $\ell_i'$ basis elements $\bm{\lambda_1^i}, \ldots, \bm{\lambda_{\ell_i'}^i}$ for $\widehat{\mathit{CFA}}(Y, \phi_r) \cdot \iota_2$ (thought of as vertices labelled by $\circ$) and differentials
\[
\stackrel{\bm{w_i'}}{\bullet} \xrightarrow{\textrm{$\rho_1$}} \stackrel{\bm{\lambda_1^i}}{\circ} \xrightarrow{\textrm{$\rho_{21}$}} \ldots \xrightarrow{\textrm{$\rho_{21}$}} \stackrel{\bm{\lambda_{\ell_i'}^i}}{\circ} \xrightarrow{\textrm{$\rho_2$}} \stackrel{\bm{w_{i+1}'}}{\bullet}.    
\]
\indent The graph of $\widehat{\mathit{CFA}}(Y, \phi_r)$ contains one more component called the \textit{unstable chain} running from the generator $\bm{w_0}$ of the vertical complex to the generator $\bm{w_0'}$ of the horizontal complex. What this looks like depends on the integer $2\tau(K)-r$, where $\tau(K)$ is an integer-valued invariant of $K$ due to Ozsv\'{a}th and Szab\'{o} in \cite{OS4} (for a quick explanation see Section \ref{epsiloninvt}).

\begin{itemize}
\item Suppose $r < 2 \tau(K)$. Let $d = 2 \tau(K) - r >0$. Then we introduce $d$ basis elements $\bm{\gamma_1}, \ldots, \bm{\gamma_d}$ for $\widehat{\mathit{CFA}}(Y, \phi_r) \cdot \iota_2$ (thought of as vertices labelled by $\circ$) and differentials
\[
\stackrel{\bm{w_0}}{\bullet} \xrightarrow{\textrm{$\rho_3$}} \stackrel{\bm{\gamma_1}}{\circ} \xleftarrow{\textrm{$\rho_{21}$}} \ldots \xleftarrow{\textrm{$\rho_{21}$}} \stackrel{\bm{\gamma_{d}}}{\circ} \xleftarrow{\textrm{$\rho_1$}} \stackrel{\bm{w_0'}}{\bullet}.
\]
\item Suppose $r > 2 \tau(K)$. Let $d = r - 2 \tau(K) >0$. Then we introduce $d$ basis elements $\bm{\gamma_1}, \ldots, \bm{\gamma_d}$ for $\widehat{\mathit{CFA}}(Y, \phi_r) \cdot \iota_2$ (thought of as vertices labelled by $\circ$) and differentials
\[
\stackrel{\bm{w_0}}{\bullet} \xrightarrow{\textrm{$\rho_{321}$}} \stackrel{\bm{\gamma_1}}{\circ} \xrightarrow{\textrm{$\rho_{21}$}} \ldots \xrightarrow{\textrm{$\rho_{21}$}} \stackrel{\bm{\gamma_{d}}}{\circ} \xrightarrow{\textrm{$\rho_2$}} \stackrel{\bm{w_{0}'}}{\bullet}.
\]
\item Finally suppose $r = 2 \tau(K)$. Then the unstable chain from $\bm{w_0}$ to $\bm{w_0'}$ takes the form
\[
\stackrel{\bm{w_0}}{\bullet} \xrightarrow{\textrm{$\rho_{32}$}} \stackrel{\bm{w_0'}}{\bullet}.
\] 
\end{itemize}

\noindent Note that $\widehat{\mathit{CFA}}(Y, \phi_r) \cdot \iota_2$ has $\mathbb{Z}_2$-dimension $(\sum_i \ell_i + \ell_i') + |2 \tau(K) - r|$ and that the elements $\bm{\kappa_e^i}, \bm{\lambda_f^i}$, and $\bm{\gamma_g}$ introduced above form a basis for $\widehat{\mathit{CFA}}(Y, \phi_r) \cdot \iota_2$.

\subsection{The knot invariant $\varepsilon$}\label{epsiloninvt}

In \cite[Section 3]{Hom}, Hom defined a $\{-1, 0, 1\}$-valued invariant $\varepsilon(K)$ for knots $K \subset S^3$ in terms of $\tau(K)$ and two other knot invariants $\nu(K)$ \cite{OS10} and $\nu'(K)$ \cite{Hom} coming from the knot Floer complex $\mathit{CFK}^{\infty}(K)$ for $K$. In this subsection we recall the definition of $\varepsilon(K)$. Throughout, we'll think of $\mathit{CFK}^{\infty}(K)$, with its two $\mathbb{Z}$-filtrations $I$ and $A$, as a directed graph in $\mathbb{Z} \times \mathbb{Z}$, with $I$ represented by the first component and $A$ by the second.

Given $S \subseteq \mathbb{Z} \times \mathbb{Z}$, one can consider the free $\mathbb{Z}_2$-vector space $C\{S\}$ generated by $S \cap \mathit{CFK}^{\infty}(K)$.  Suppose $S$ has the property that every point in $\mathbb{Z} \times \mathbb{Z}$ that's either to the left or below some point in $S$ is already an element of $S$, in other words $S$ is closed under the operations of looking down and to the left. Then $C\{S\}$, together with the differential induced by $\partial^{\infty}$, gives us a $\mathbb{Z}_2$-chain complex. When $S_1$ and $S_2$ are two subsets of $\mathbb{Z} \times \mathbb{Z}$ with the above property, and $S_1 \supseteq S_2$, we can form the quotient chain complex $C\{S_1\} / C\{S_2\}$. 

We define $\tau(K)$ to be the minimum Alexander filtration level $s$ so that the inclusion map
\[
\textrm{incl}: C\{I \leq 0, A \leq s\} / C\{I < 0, A \leq s\} \hookrightarrow C\{I \leq 0\} / C\{I < 0\} \simeq \widehat{\mathit{CF}}(S^3)
\]
of $\mathbb{Z}_2$-chain complexes induces a non-trivial map on homology.

The invariants $\nu(K)$ and $\nu'(K)$ come from studying more complicated regions of the $\mathit{CFK}^{\infty}(K)$ graph. $\forall s \in \mathbb{Z}$, let $A_s$ be the $\mathbb{Z}_2$-vector space
\[
C\{\textrm{max}(I, A-s) \leq 0 \} / C\{\textrm{max}(I, A-s) <0 \}
\]
\noindent and let $A_s'$ be the $\mathbb{Z}_2$-vector space
\[
C\{\textrm{min}(I, A-s) \leq 0\} / C\{\textrm{min}(I, A-s) < 0\}.
\]
\noindent By equipping $A_s$ and $A_s'$ with the differentials induced by $\partial^{\infty}$, we can think of $A_s$ and $A_s'$ as $\mathbb{Z}_2$-chain complexes. Like we did for $\tau(K)$, we have chain maps 
\[
\nu_s: A_s \rightarrow \widehat{\mathit{CF}}(S^3) 
\]
and
\[
\nu_s': \widehat{\mathit{CF}}(S^3) \rightarrow A_s'
\]
\noindent given as follows: $\nu_s$ is the composition
\[
A_s \xrightarrow{\textrm{quot}} C\{I \leq 0, A \leq s\} / C\{I < 0, A \leq s\} \xhookrightarrow{\textrm{incl}} C\{I \leq 0\} / C\{I < 0\} \simeq \widehat{\mathit{CF}}(S^3)
\]
\noindent and $\nu_s'$ is the composition
\[
\widehat{\mathit{CF}}(S^3) \simeq C\{I \leq 0\} / C\{I < 0\} \xrightarrow{\textrm{quot}} C\{I \leq 0\} / C\{(I <0) \cup (I =0, A < s)\} \xhookrightarrow{\textrm{incl}} A_s'.  
\]
We define the invariant $\nu(K)$ to be the minimum Alexander filtration level $s$ so that the chain map $\nu_s$ induces a nontrivial map on homology, and the invariant $\nu '(K)$ to be the maximum Alexander filtration level $s$ so that the chain map $\nu_s'$ induces a nontrivial map on homology. Then the invariant $\varepsilon(K)$ is the integer $2\tau(K) - \nu(K) - \nu '(K)$. That $\varepsilon(K) \in \{-1, 0, 1\}$ is due to Hom in \cite[Lemmas 3.2 and 3.3]{Hom}. 

\section{$\widehat{\mathit{HFO}}(Y^{\textrm{orb}})$: Definition, Theorem  \ref{definvariant}, and Examples}\label{Sec: Orbifold Invt}

\subsection{Definition of $\widehat{\mathit{HFO}}(Y^{\textrm{orb}})$}

Let $Y^{\textrm{orb}}$ be a compact, connected, oriented 3-orbifold with singular set a knot $K$ of order $n$. Fix a neighborhood $N$ of $K$ modeled on $(D^2 \times S^1) / \mathbb{Z}_n$ and an orientation-preserving homeomorphism $\phi_N: (D^2 \times S^1) / \mathbb{Z}_n  \rightarrow N$. What will be important for us is the induced orientation-preserving parameterization of the boundary:
\[
\phi_{\partial N}: \partial \big( (D^2 \times S^1) / \mathbb{Z}_n \big) \rightarrow \partial N.
\]
There's a natural orientation-reversing identification of the oriented torus $F$ associated to the pointed matched circle $Z$ from Figure \ref{fig:PointedMatchedCircle} with $\partial \big( (D^2 \times S^1) / \mathbb{Z}_n \big)$, taking $\alpha_1^a$  to the longitude $\{ \overline{1} \} \times S^1$ and $\alpha_2^a$ to the meridian $\partial D^2 / \mathbb{Z}_n \times \{1\}$. This allows us to view $\phi_{\partial N}$ as an orientation-reversing parameterization of $\partial N$ by $F$.

If we remove (the interior of) the singular neighborhood $N$, we're left with an honest 3-manifold $E$ with torus boundary. Using the orientation-reversing parameterization $\phi_{\partial N}$ of $\partial N$, we can define the following orientation-preserving parameterization $\phi_{\partial E}$ of $\partial E$:
\[
\phi_{\partial E}: = id \circ \phi_{\partial N} : F \rightarrow \partial E.
\]
Then $E$, together with $\phi_{\partial E}$, forms a type A bordered 3-manifold. To $(E, \phi_{\partial E})$ we associate the type A structure $\widehat{\mathit{CFA}}(E, \phi_{\partial E})$ coming from bordered Floer theory.

Generalizing the type D structure $\widehat{\mathit{CFD}}(D^2 \times S^1, \psi)$ in Figure \ref{fig:CFD_Solid_Torus}, we associate to the singular piece $N$ the type D structure $D_N$ in Figure \ref{fig:CFD_Orbifold_Solid_Torus}.

\begin{figure}[h]
\centering{
\resizebox{60mm}{!}{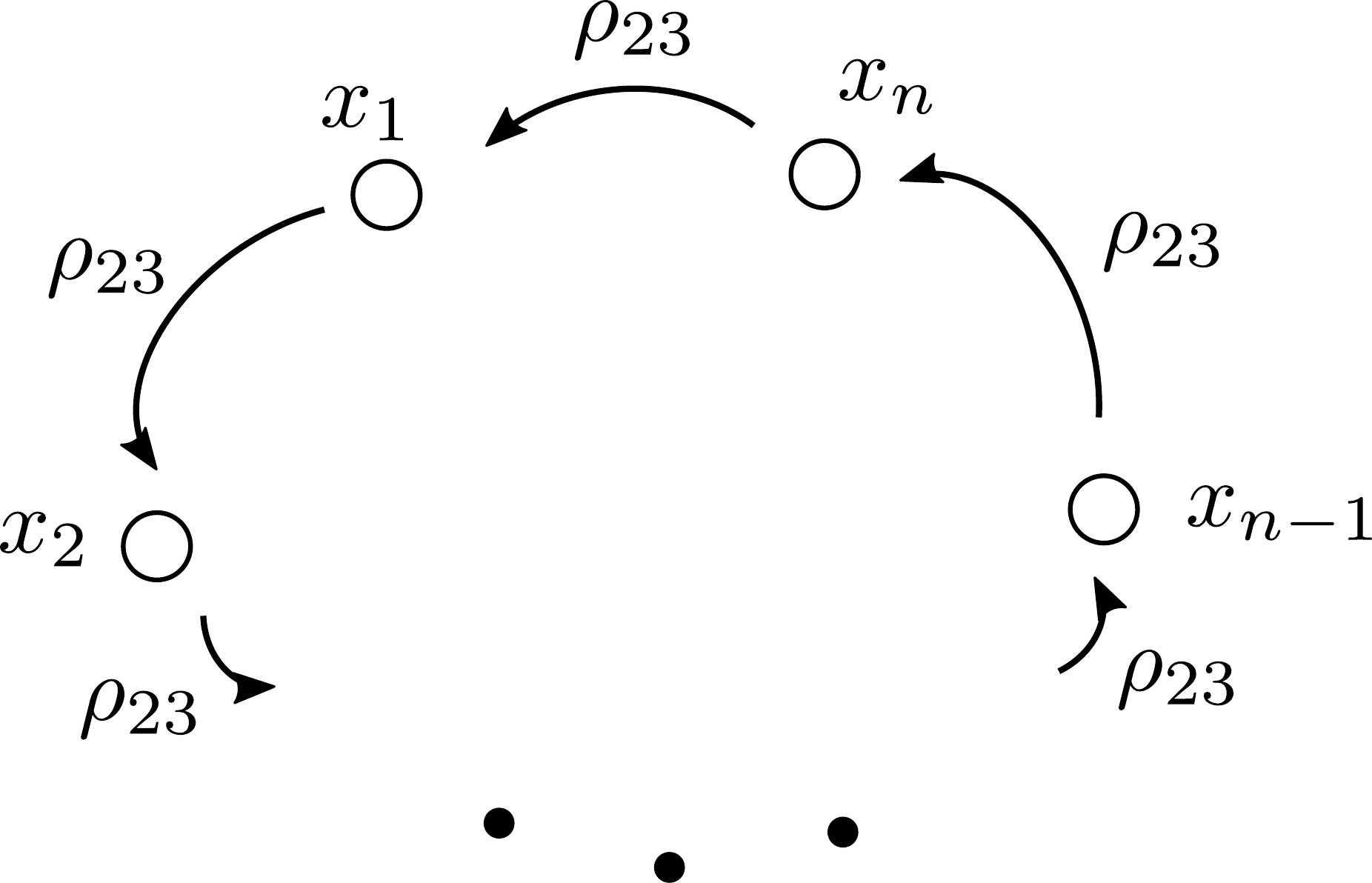}
\caption{The type D structure $D_N$}
\label{fig:CFD_Orbifold_Solid_Torus}
}
\end{figure}

Similar to $\widehat{\mathit{CFD}}(D^2 \times S^1, \psi)$, $D_N$ arises naturally from an ``orbifold bordered Heegaard diagram" for $(N, \phi_{\partial N})$; see Figure \ref{fig:OrbifoldBHD}. Here we're starting with a $\mathbb{Z}_n$-equivariant torus that has been punctured once, together with two properly embedded arcs $\alpha_1^a$ and $\alpha_2^a$. When we fill in the puncture, we recover $\partial N$. Like before, $\beta$ represents a meridian of an honest handlebody, but unlike before, $\beta$ sits immersed in the punctured $\mathbb{Z}_n$-equivariant torus, wrapping $n$ times around $\alpha_2^a$ because $\beta$ represents one full meridian, while $\alpha_2^a$ represents a meridian of the $\mathbb{Z}_n$-equivariant solid torus $N$, i.e. an $n$th of a full meridian. The generators $x_i$ of the type D structure $D_N$ correspond to where the $\beta$ curve intersects the $\alpha^a_1$ arc. The differential corresponds to counting domains with corners only at the generators. For an example of a domain that doesn't contribute to the differential, see Figure \ref{fig:Bad_Orbifold_Domain}.

\begin{figure}[h]
    \centering
    \begin{subfigure}[t]{0.5\textwidth}
        \centering
        \includegraphics[height=1.5in]{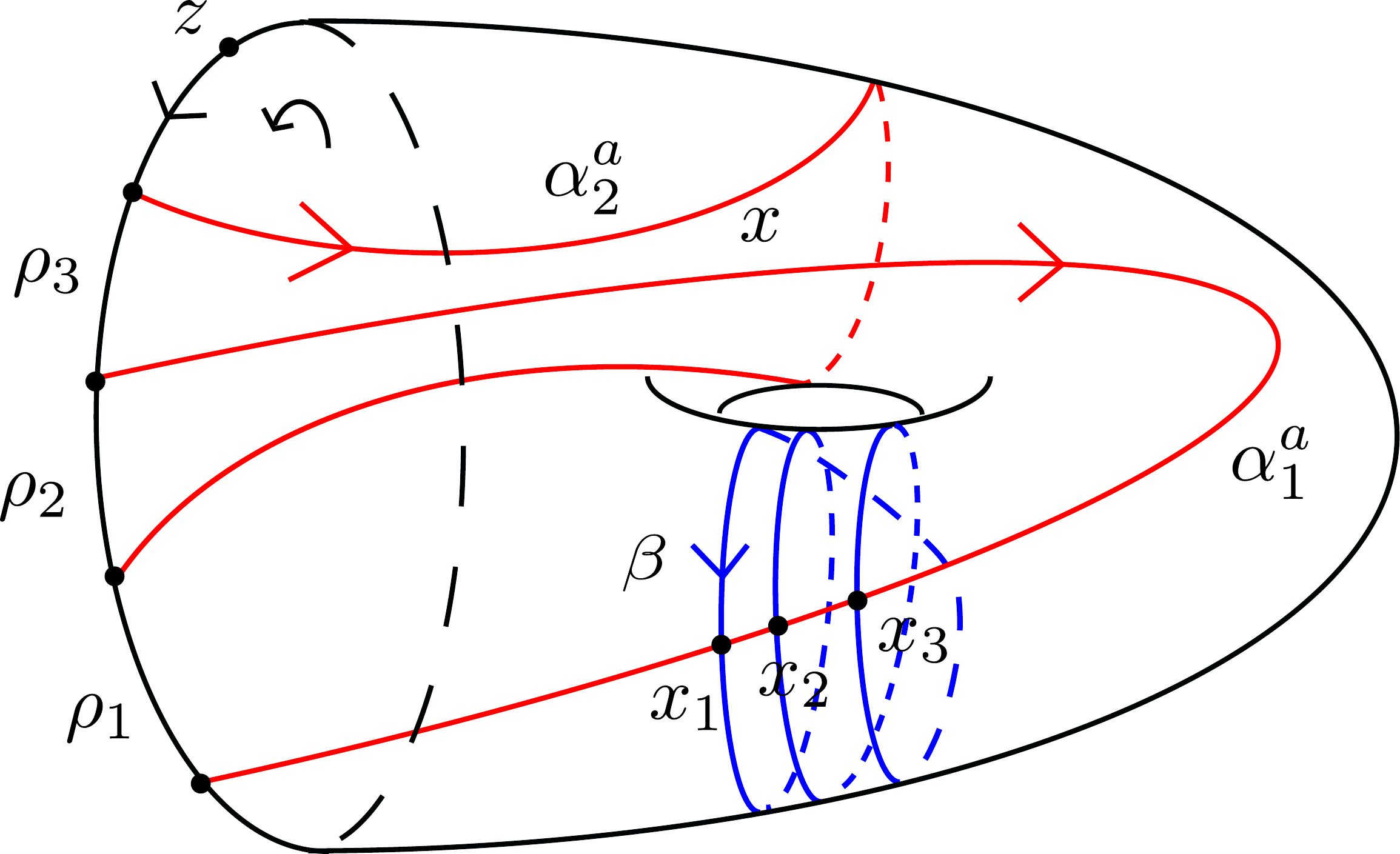}
        \caption{}
    \end{subfigure}%
    ~ 
    \begin{subfigure}[t]{0.5\textwidth}
        \centering
        \includegraphics[height=1.5in]{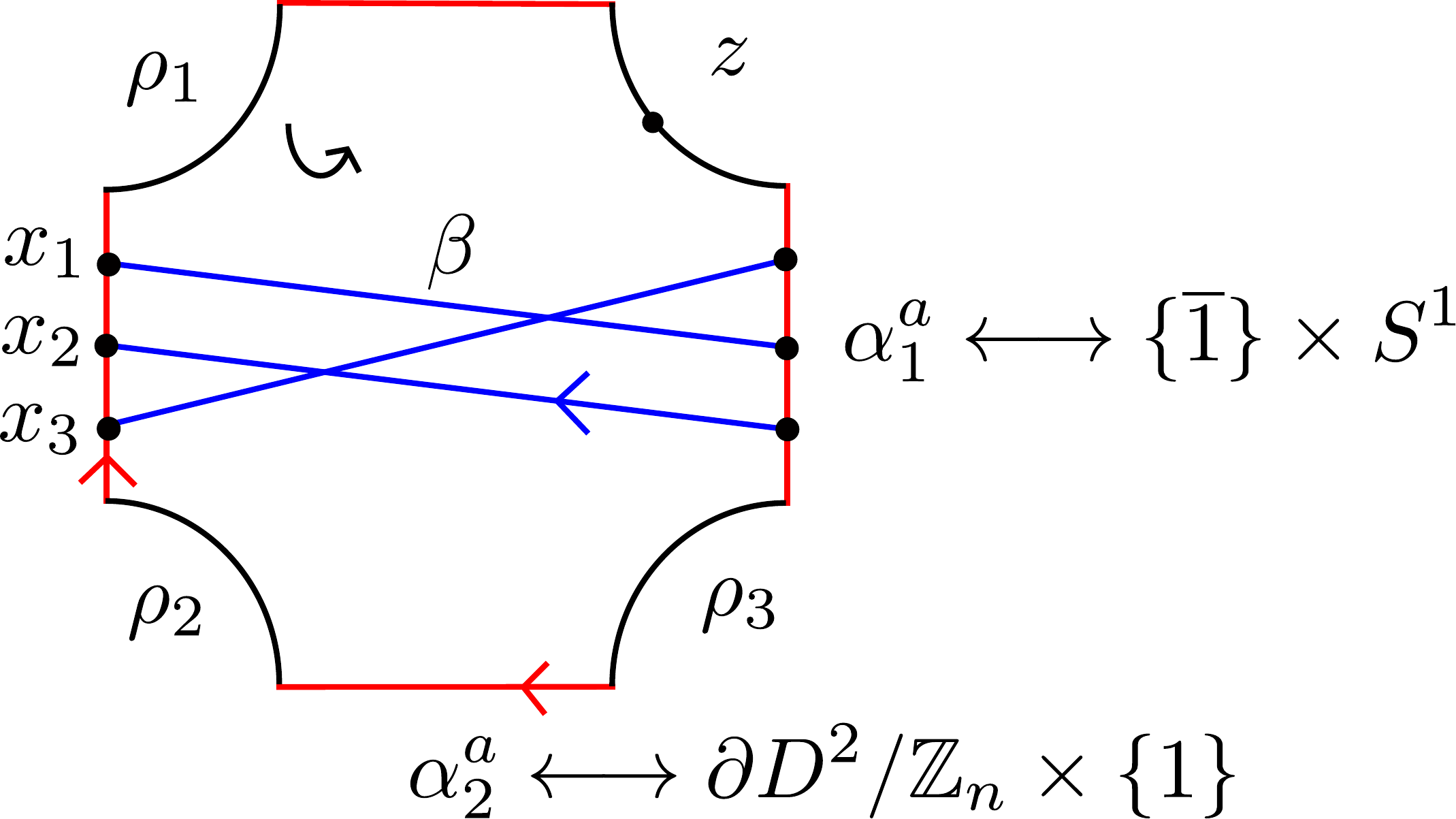}
        \caption{}
    \end{subfigure}
    \caption{Two ways to describe a genus 1 orbifold bordered Heegaard diagram for $(N, \phi_{\partial N})$ that yields $D_N$ in the case when $n=3$. The other values of $n$ are similar. Compare with Figure \ref{fig:BHD}.} 
    \label{fig:OrbifoldBHD}
\end{figure}

\begin{figure}[h]
    \centering
    \includegraphics[height=1.5in]{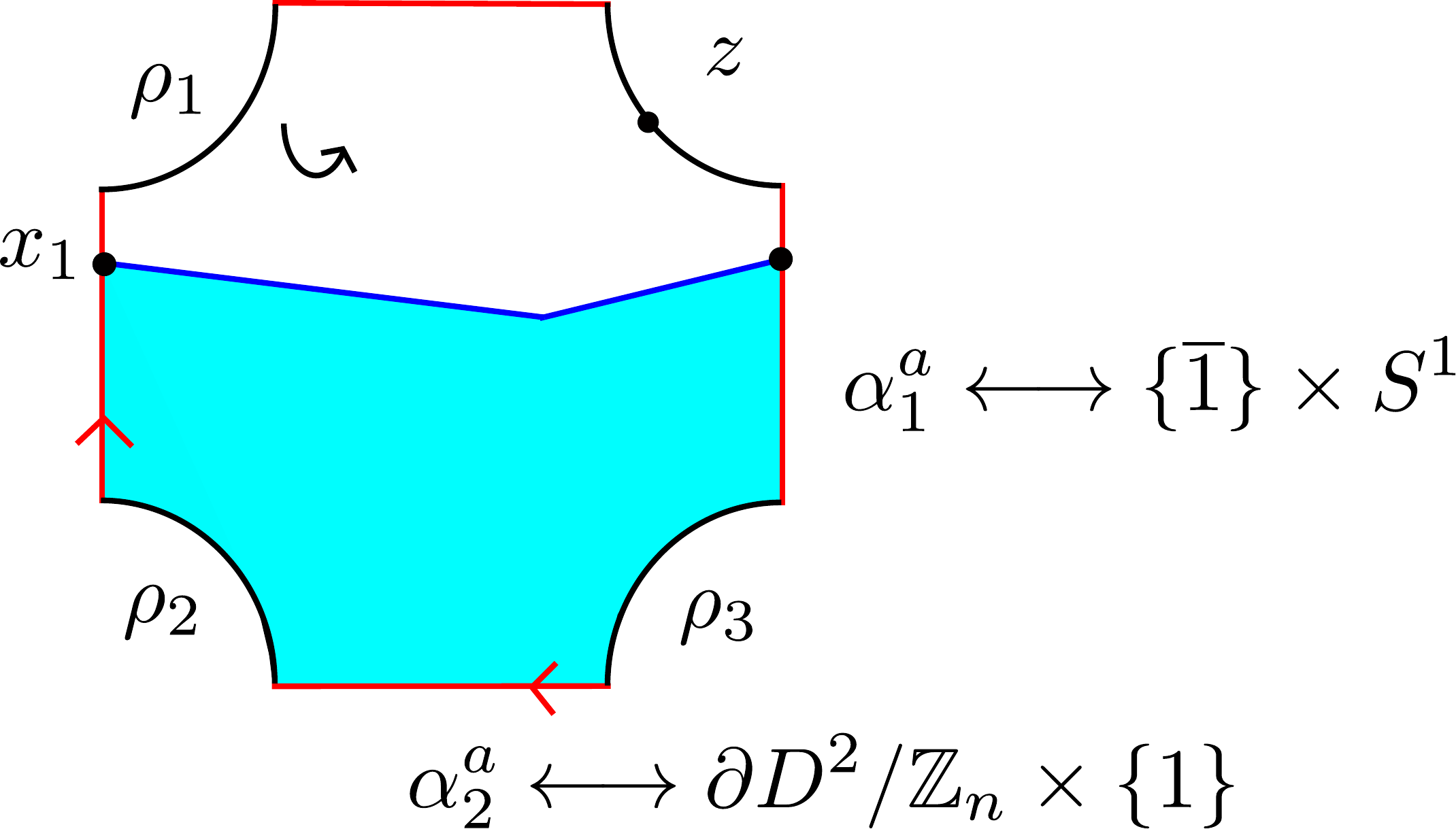}
            \caption{A domain that doesn't get counted towards the type D structure $D_N$ in the case when $n=3$}
\label{fig:Bad_Orbifold_Domain}
\end{figure}

\begin{Remark}
The type D structure $D_N$ isn't bounded, but by performing a ``finger move" on one of the edges we can pass to a homotopy equivalent type D structure that is bounded. See Figure \ref{fig:bounded_version} for an example.
\end{Remark}

\begin{figure}[h]
    \centering
    \begin{subfigure}[t]{0.4\textwidth}
        \centering
        \includegraphics[height=1.5in]{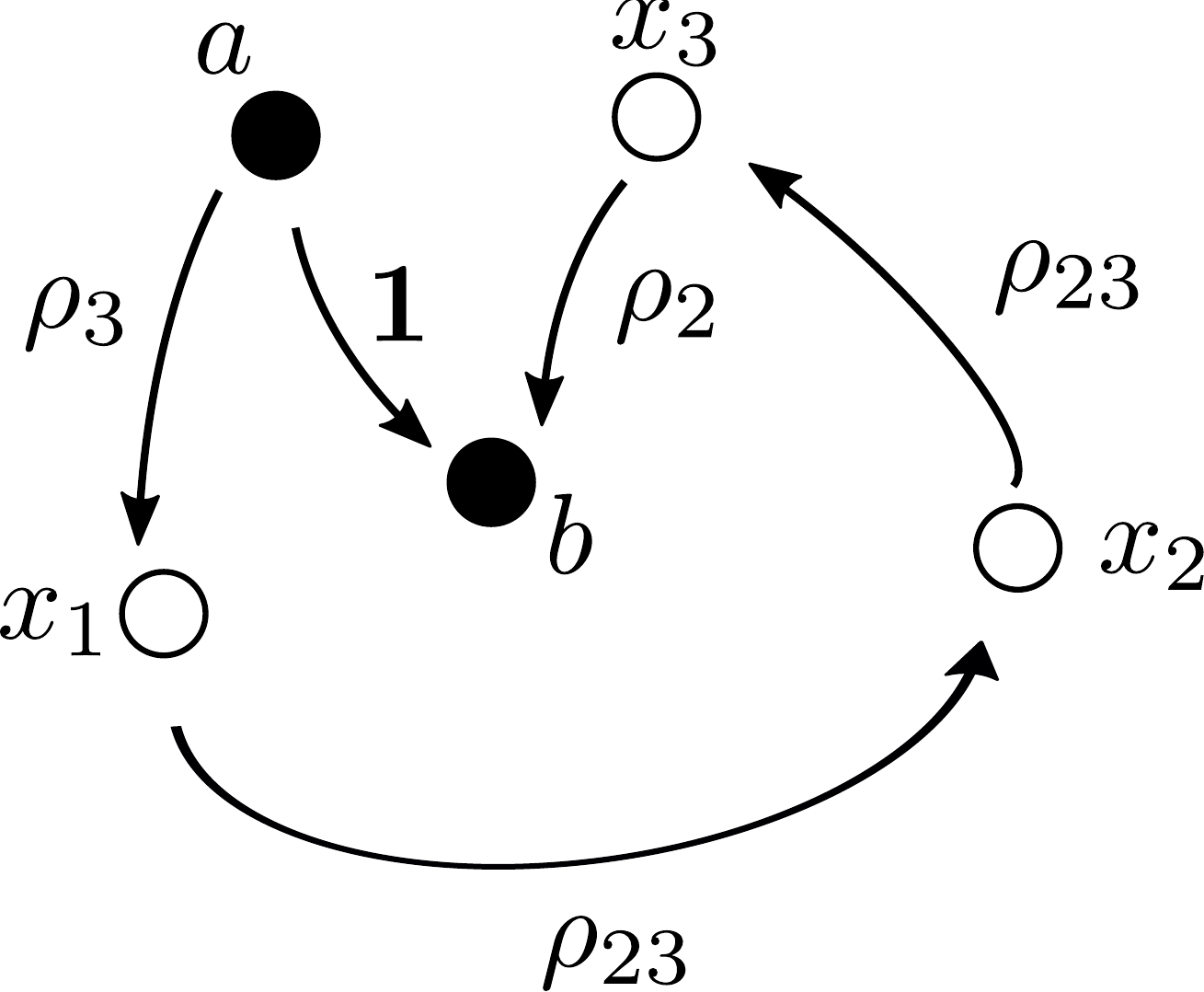}
        \caption{}
    \end{subfigure}%
    ~ 
    \begin{subfigure}[t]{0.4\textwidth}
        \centering
        \includegraphics[height=1.7in]{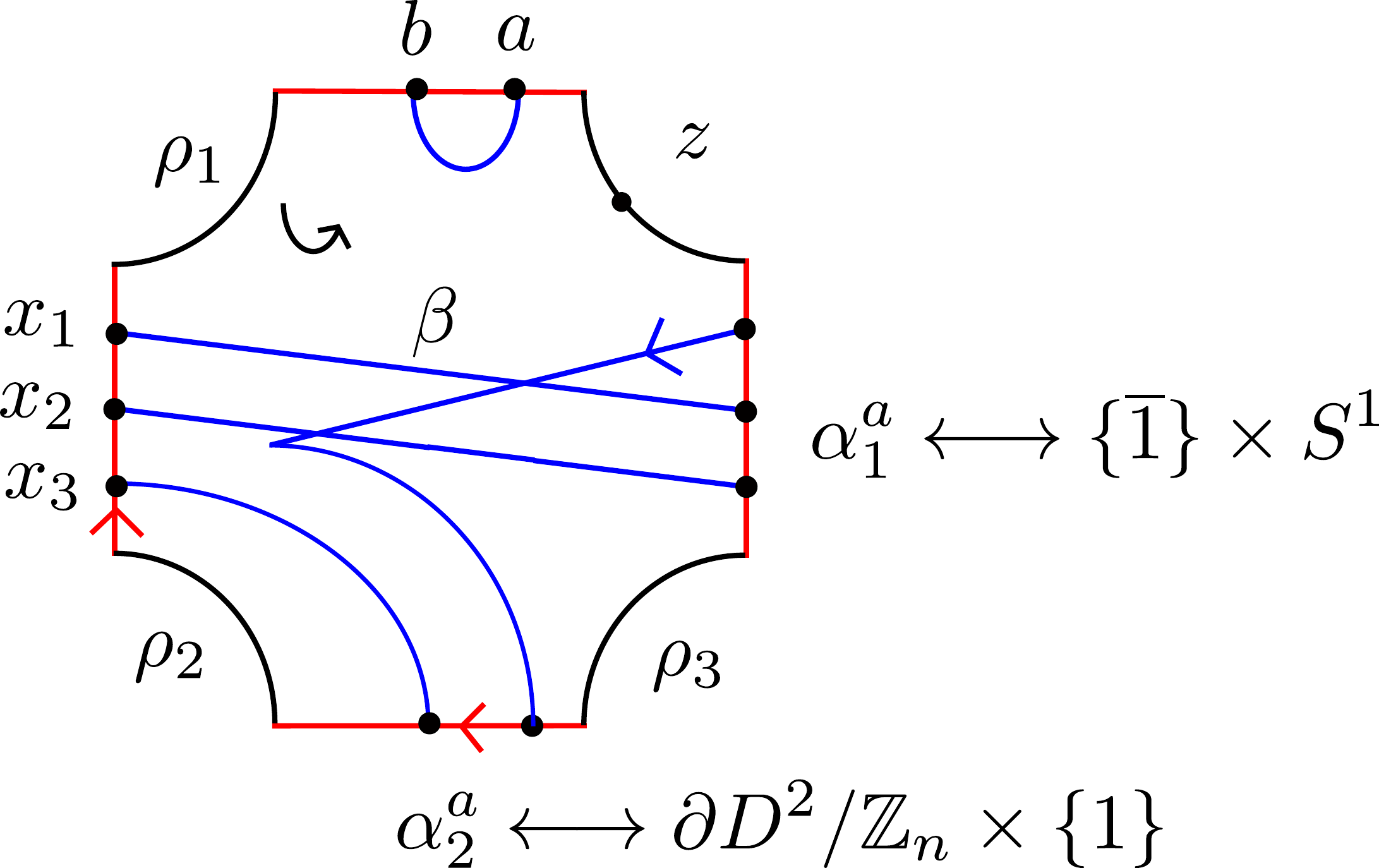}
        \caption{}
    \end{subfigure}
\caption{A bounded type D structure that is homotopy equivalent to $D_N$ and an orbifold bordered Heegaard diagram for $(N, \phi_{\partial N})$ that gives rise to it in the case when $n=3$}
\label{fig:bounded_version}
\end{figure}

\begin{Def}
Let $\widehat{\mathit{CFO}}(Y^{\textrm{orb}})$ be the box tensor product $\widehat{\mathit{CFA}}(E, \phi_{\partial E}) \boxtimes D_N$. We define $\widehat{\mathit{HFO}}(Y^{\textrm{orb}})$ to be the homology of $\widehat{\mathit{CFO}}(Y^{\textrm{orb}})$.
\end{Def}

\begin{Remark}
$\widehat{\mathit{CFA}}(E, \phi_{\partial E}) \boxtimes D_N$ only makes sense for bounded $\widehat{\mathit{CFA}}(E, \phi_{\partial E})$. When $\widehat{\mathit{CFA}}(E, \phi_{\partial E})$ isn't bounded, we consider $\widehat{\mathit{CFA}}(E, \phi_{\partial E}) \boxtimes D_N'$ instead, where $D_N'$ is any type D structure obtained from $D_N$ by a finger move as described above. Note that $\widehat{\mathit{CFA}}(E, \phi_{\partial E}) \boxtimes D_N$ and $\widehat{\mathit{CFA}}(E, \phi_{\partial E}) \boxtimes D_N'$ are homotopy equivalent, so we haven't lost anything by passing to $\widehat{\mathit{CFA}}(E, \phi_{\partial E}) \boxtimes D_N'$.
\end{Remark}
  
\subsection{Proof of Theorem \ref{definvariant}} Here we prove that $\widehat{\mathit{HFO}}(Y^{\textrm{orb}})$ is a well-defined invariant of $Y^{\textrm{orb}}$ that generalizes $\widehat{\mathit{HF}}$ for 3-manifolds.

\begin{proof}[Proof that $\widehat{\mathit{HFO}}(Y^{\emph{orb}})$ is well-defined] We need to show that $\widehat{\mathit{HFO}}(Y^{\textrm{orb}})$ is independent of the equivariant neighborhood $N$ and the orientation-preserving parameterization $\phi_N: (D^2 \times S^1) / \mathbb{Z}_n  \rightarrow N$. First we argue that for a fixed neighborhood $N$, $\widehat{\mathit{HFO}}(Y^{\textrm{orb}})$ is independent of the parameterization $\phi_N$. Let $\phi_N ^1$ and $\phi_N ^2$ be two orientation-preserving parameterizations of $N$ by $(D^2 \times S^1) / \mathbb{Z}_n$. From $\phi_N ^1$ and $\phi_N ^2$, we get two type A bordered 3-manifolds $(E, \phi_{\partial E}^1)$ and $(E, \phi_{\partial E}^2)$. It suffices to show that the resulting $\mathbb{Z}_2$-chain complexes $\widehat{\mathit{CFA}}(E, \phi_{\partial E} ^1) \boxtimes D_{N}$ and $\widehat{\mathit{CFA}}(E, \phi_{\partial E} ^2) \boxtimes D_{N}$ are chain homotopy equivalent. Because $\widehat{\mathit{CFA}}(E, \phi_{\partial E} ^1)$ and $\widehat{\mathit{CFA}}(E, \phi_{\partial E} ^2)$ may not be bounded, we'll need to replace $D_N$ with a bounded type D structure $D_N'$ that's homotopy equivalent to $D_N$; we'll use the one in Figure \ref{fig:Bounded_DN}.

\begin{figure}[h]
\centering{
\resizebox{70mm}{!}{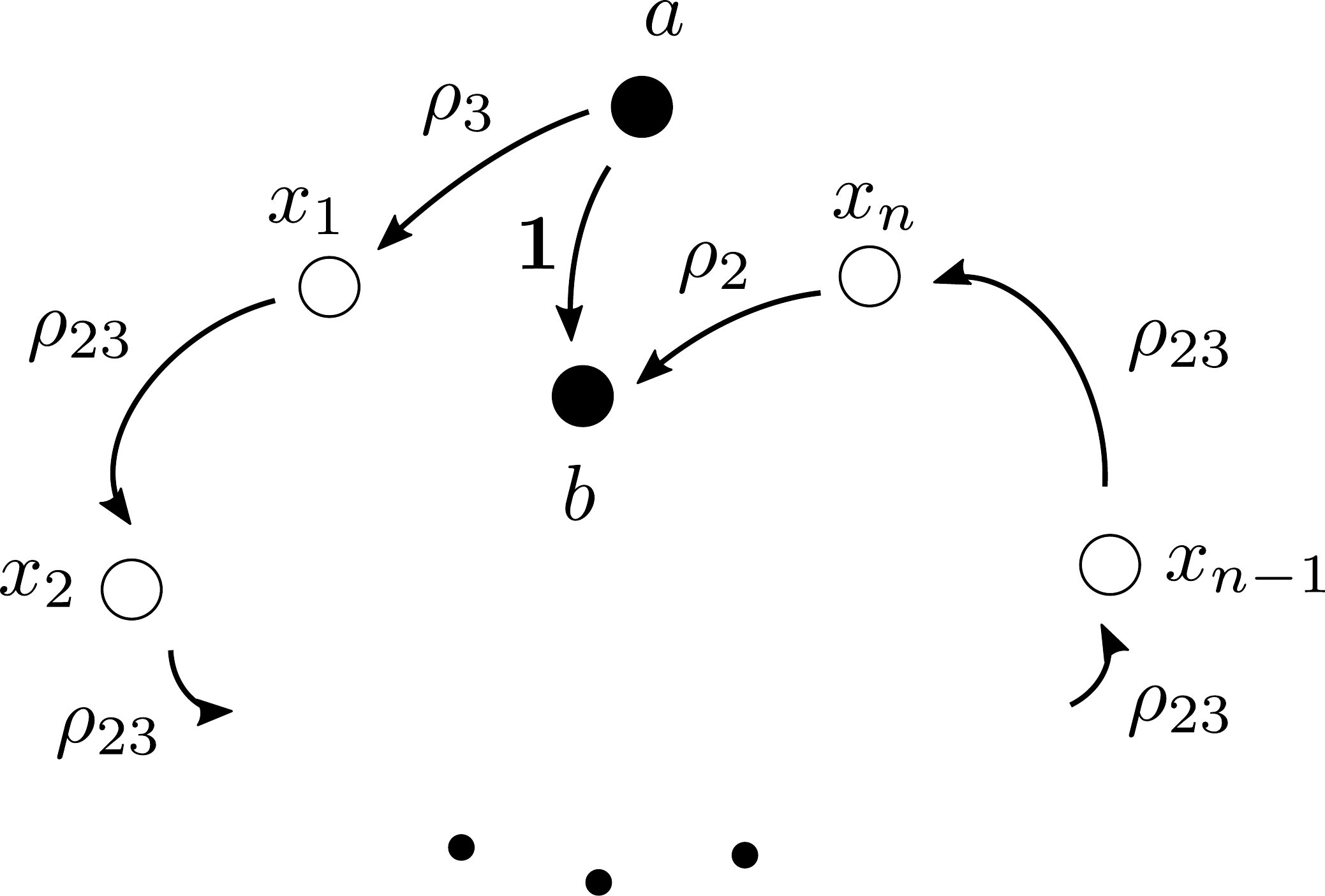}
\caption{A bounded type D structure $D_N'$ homotopy equivalent to $D_N$}
\label{fig:Bounded_DN}
}
\end{figure}

We claim that the $\mathbb{Z}_2$-chain complexes $\widehat{\mathit{CFA}}(E, \phi_{\partial E} ^1) \boxtimes D_{N}'$ and $\widehat{\mathit{CFA}}(E, \phi_{\partial E} ^2) \boxtimes D_{N}'$ are chain homotopy equivalent. We prove this as follows. Let $\psi: F \rightarrow F$ be the composition $ \big( (\phi_{\partial E} ^1)^{-1} \circ \phi_{\partial E} ^2 \big)^{-1}$. Note that $\phi_{\partial E} ^2 = \phi_{\partial E} ^1 \circ \psi ^{-1}$. By Equation \ref{changeofparameqn},
\[
\widehat{\mathit{CFA}}(E, \phi_{\partial E} ^2) \simeq \widehat{\mathit{CFA}}(E, \phi_{\partial E} ^1)  \boxtimes \widehat{\mathit{CFDA}}(\psi).
\]
Then
\begin{equation*}
\begin{split}
\widehat{\mathit{CFA}}(E, \phi_{\partial E} ^2) \boxtimes D_{N}' & \simeq \Big( \widehat{\mathit{CFA}}(E, \phi_{\partial E} ^1)  \boxtimes \widehat{\mathit{CFDA}}(\psi) \Big) \boxtimes D_N' \\
& \simeq \widehat{\mathit{CFA}}(E, \phi_{\partial E} ^1)  \boxtimes \Big( \widehat{\mathit{CFDA}}(\psi) \boxtimes D_N' \Big),
\end{split}
\end{equation*}
so if we can show
\begin{equation*}
\widehat{\mathit{CFDA}}(\psi) \boxtimes D_N' \simeq D_N ',
\end{equation*}

\noindent then we have the claim.

\begin{Lem}
Let $\tau_{\alpha ^a_2} : F \rightarrow F$ be the Dehn twist about the curve $\alpha ^a_2$. Then $\psi$ is isotopic to a power of $\tau_{\alpha ^a_2}$.
\end{Lem}

\begin{proof}
It's enough to show that the composition $(\phi_{\partial N} ^1)^{-1} \circ \phi_{\partial N} ^2 : \partial \big((D^2 \times S^1) / \mathbb{Z}_n\big) \rightarrow \partial \big((D^2 \times S^1) / \mathbb{Z}_n\big) $ is isotopic to the Dehn twist about the meridian $\partial D^2 / \mathbb{Z}_n \times \{1\} $, which we denote by $m$ for convenience. By construction, $(\phi_{\partial N} ^1)^{-1} \circ \phi_{\partial N} ^2$ extends to a homeomorphism of $(D^2 \times S^1) / \mathbb{Z}_n$. Then $(\phi_{\partial N} ^1)^{-1} \circ \phi_{\partial N} ^2$ has to send $m$ to a meridian of $(D^2 \times S^1) / \mathbb{Z}_n$, which means that $(\phi_{\partial N} ^1)^{-1} \circ \phi_{\partial N} ^2$ is isotopic to a power of $D_m$.
\end{proof}

\noindent We can assume $\psi \simeq (\tau_{\alpha ^a_2})^n$ for some $n \in \mathbb{N}$ because there's a similar argument for $\psi \simeq  (\tau_{\alpha ^a_2}^{-1})^{n}$. From \cite[Theorem 5]{LOT2},
\[
\widehat{\mathit{CFDA}}(\psi) \simeq \underbrace{\widehat{\mathit{CFDA}}(\tau_{\alpha ^a_2}) \boxtimes \cdots \boxtimes \widehat{\mathit{CFDA}}(\tau_{\alpha ^a_2})}_\text{n times},
\] 
so it suffices to show
\begin{equation}\label{NTS2}
\widehat{\mathit{CFDA}}(\tau_{\alpha ^a_2}) \boxtimes D_N ' \simeq D_N '.
\end{equation}

From \cite[Proposition 10.6]{LOT2}, $\widehat{\mathit{CFDA}}(\tau_{\alpha ^a_2})$ \big(which Lipshitz, Ozsv\'{a}th, and D. Thurston call $\widehat{\mathit{CFDA}}(\tau_{m},0)$\big) has generators $\textbf{p}$, $\textbf{q}$, and $\textbf{r}$, with the non-zero $(\mathcal{I}, \mathcal{I})$-actions given by
\begin{align*}
\iota_1 \cdot \textbf{p} \cdot \iota_1 &= \textbf{p} & \iota_2 \cdot \textbf{q} \cdot \iota_2 &= \textbf{q} & \iota_2 \cdot \textbf{r} \cdot \iota_1 &= \textbf{r},
\end{align*}

\noindent and the non-trivial differentials given by
\begin{align*}
\delta^2_1(\textbf{p}, \rho_1) &= \rho_1 \otimes \textbf{q} & \delta^2_1(\textbf{p}, \rho_{12}) &= \rho_{123} \otimes \textbf{r}\\
\delta^2_1(\textbf{p}, \rho_{123}) &= \rho_{123} \otimes \textbf{q} & \delta^3_1(\textbf{p}, \rho_3, \rho_2) &= \rho_3 \otimes \textbf{r}\\
\delta^3_1(\textbf{p}, \rho_3, \rho_{23}) &= \rho_3 \otimes \textbf{q} & \delta^2_1(\textbf{q}, \rho_2) &= \rho_{23} \otimes \textbf{r}\\
\delta^2_1(\textbf{q}, \rho_{23}) &= \rho_{23} \otimes \textbf{q} & \delta^2_1(\textbf{r}) &= \rho_2 \otimes \textbf{p}\\
\delta^2_1(\textbf{r}, \rho_3) &= \textbf{1} \otimes \textbf{q}. & 
\end{align*}

\noindent By direct computation, we get that the type D structure $\widehat{\mathit{CFDA}}(\tau_{\alpha ^a_2}) \boxtimes D_N '$ is given by the decorated, directed graph in Figure \ref{fig:10}. If we cancel the edges $\stackrel{\bm{p} \otimes a}{\bullet} \xrightarrow{\textrm{$\bm{1}$}} \stackrel{\bm{p} \otimes b}{\bullet}$ and $\stackrel{\bm{r} \otimes a}{\circ} \xrightarrow{\textrm{$\bm{1}$}} \stackrel{\bm{r} \otimes b}{\circ}$ as prescribed by the well-known ``edge reduction" algorithm \cite{Levine}, Figure \ref{fig:10} reduces to Figure \ref{fig:CFD_Orbifold_Solid_Torus}. This shows that $\widehat{\mathit{CFDA}}(\tau_{\alpha ^a_2}) \boxtimes D_N '$ and $D_N'$ are homotopy equivalent.

\begin{figure}[h]
\centering{
\resizebox{90mm}{!}{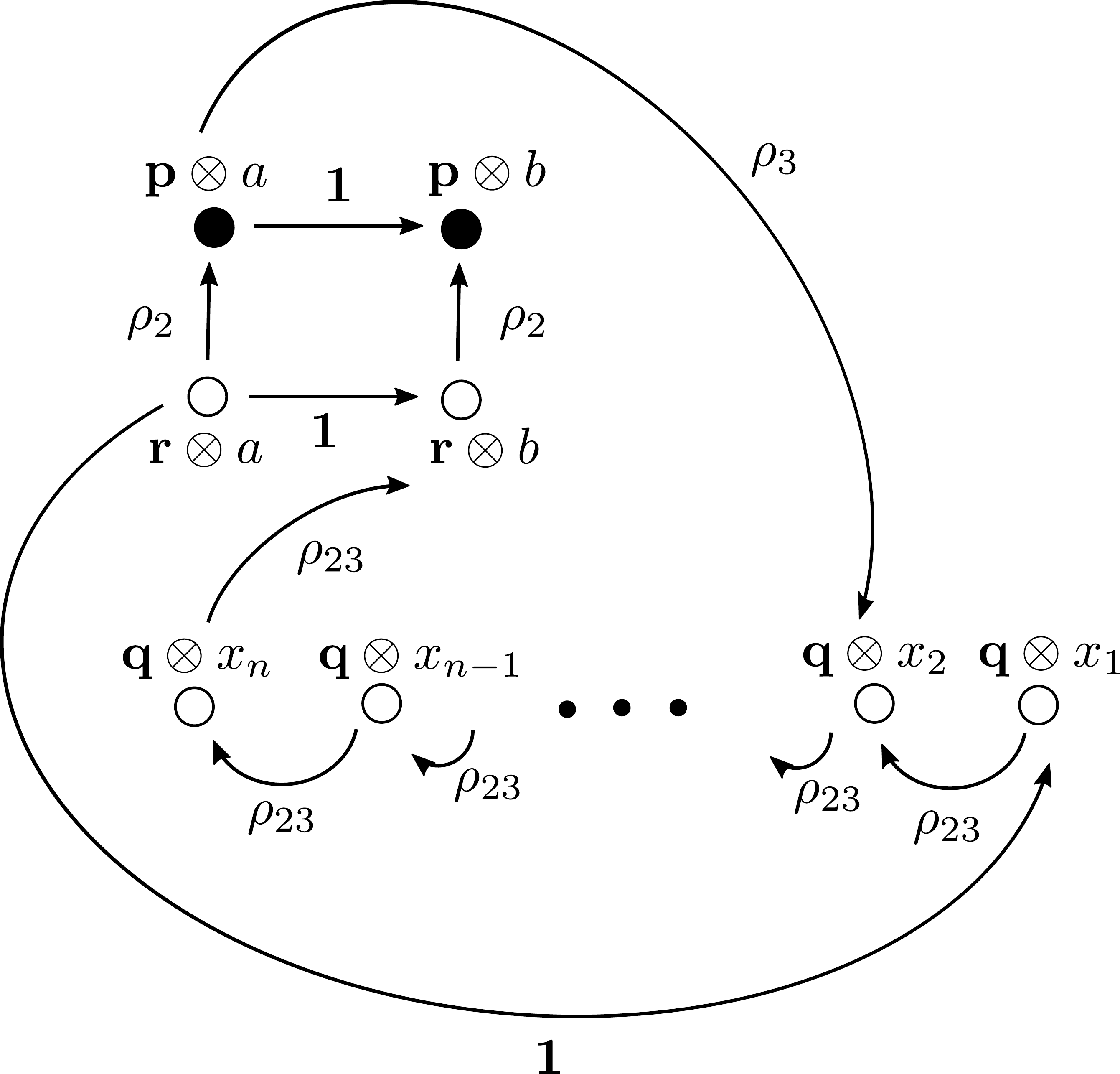}
\caption{The type D structure $\widehat{\mathit{CFDA}}(\tau_{\alpha ^a_2}) \boxtimes D_N '$}
\label{fig:10}
}
\end{figure}

Now we check that $\widehat{\mathit{HFO}}(Y^{\textrm{orb}})$ is independent of the singular tubular neighborhood $N$. Let $N_1$ and $N_2$ be two singular tubular neighborhoods of $K$. Since $D_{N_1} = D_{N_2}$, it's enough to show that the type A structures $\widehat{\mathit{CFA}}(E_1, (\phi_1)_{\partial E_1})$ and $\widehat{\mathit{CFA}}(E_2, (\phi_2)_{\partial E_2})$ are homotopy equivalent, for some choice of $\phi_1: (D^2 \times S^1) / \mathbb{Z}_n  \rightarrow N_1$  and $\phi_2: (D^2 \times S^1) / \mathbb{Z}_n  \rightarrow N_2$. Let $\phi_1: (D^2 \times S^1) / \mathbb{Z}_n  \rightarrow N_1$ be any orientation-preserving parameterization of $ N_1$. Since $N_1$ and $N_2$ are tubular neighborhoods of the same $K$, $N_1$ and $N_2$ are ambiently isotopic, so pick an ambient isotopy $H_t$ of $|Y^{\textrm{orb}}|$ that takes $N_1$ to $N_2$. Then we can define $\phi_2$ to be the composition $H_1 |_{N_1} \circ \phi_1$. By construction, we have the commutative diagram in Figure \ref{fig:11}. This implies that the bordered 3-manifolds $(E_1, (\phi_1)_{\partial E_1})$ and $(E_2, (\phi_2)_{\partial E_2})$ are equivalent, which in turn implies that $\widehat{\mathit{CFA}}(E_1, (\phi_1)_{\partial E_1})$ and $\widehat{\mathit{CFA}}(E_2, (\phi_2)_{\partial E_2})$ are homotopy equivalent. This concludes the proof that $\widehat{\mathit{HFO}}(Y^{\textrm{orb}})$ is well-defined.

\begin{figure}[h]
\centering
\begin{tikzcd}[column sep={6em,between origins}, row sep=huge]
&&& F \arrow{dr}{\cong} \arrow[swap]{dl}{\cong} &&& \\
&& \partial \big( (D^2 \times S^1) / \mathbb{Z}_n \big) \arrow[swap]{dl}{\phi_1|_{\partial}} && \partial \big( (D^2 \times S^1) / \mathbb{Z}_n \big) \arrow{dr}{\phi_2|_{\partial}} &&\\
& \partial N_1 \arrow[swap]{dl}{id} &&&& \partial N_2 \arrow{dr}{id} &\\
\partial E_1 \arrow[swap]{rrrrrr}{H_1|_{\partial E_1}} &&&&&& \partial E_2
\end{tikzcd}
\caption{$(E_1, \phi_{\partial E_1} ^1)$ and $(E_2, \phi_{\partial E_2} ^1)$ are equivalent}
\label{fig:11}
\end{figure}
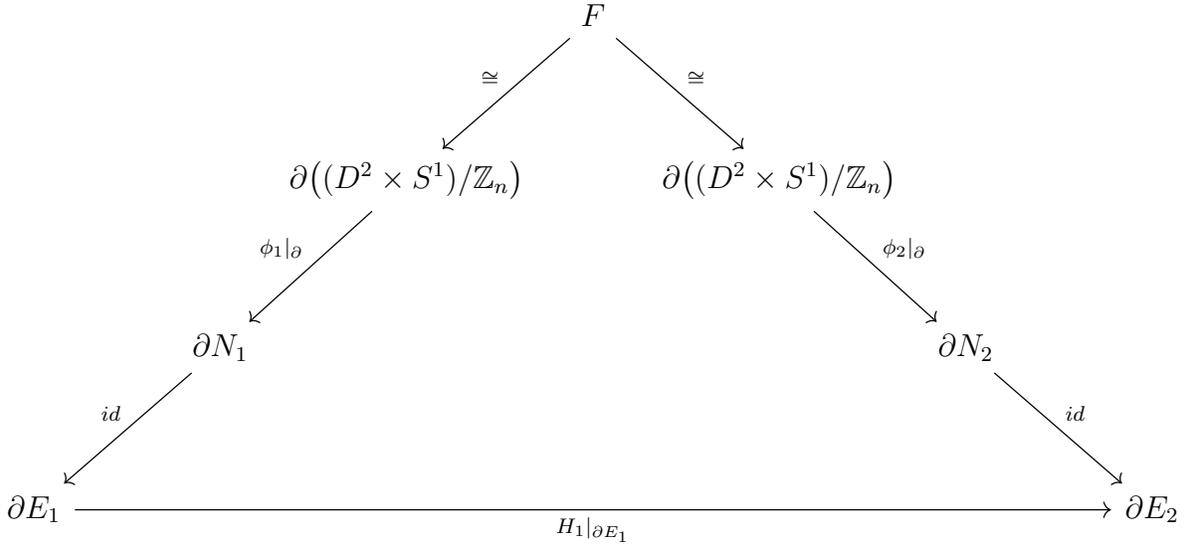
\end{proof}
 
\begin{proof}[Proof that $\widehat{\mathit{HFO}}(Y^{\emph{orb}})$ is an invariant of 3-orbifolds] Let $(Y_1, K_1, n)$ and $(Y_2, K_2, n)$ be homeomorphic (oriented) 3-orbifolds. We need to show $\widehat{\mathit{HFO}}(Y_1, K_1, n) \cong \widehat{\mathit{HFO}}(Y_2, K_2, n)$. We have an orientation-preserving homeomorphism $|f|: Y_1 \rightarrow Y_2$ between the underlying oriented 3-manifolds $Y_1$ and $Y_2$ taking a singular neighborhood $N_1$ of $K_1$ to a singular neighborhood $N_2$ of $K_2$. Let $E_i = Y_i - \textrm{int}(N_i)$. Then $|f| (E_1) = E_2$. Now pick any orientation-preserving parameterization $\phi_1 : (D^2 \times S^1) / \mathbb{Z}_n  \rightarrow N_1$ of $N_1$. As described above, we get an orientation-preserving parameterization $(\phi_1)_{\partial E_1}: F \rightarrow \partial E_1$ of $\partial E_1$. Define $(\phi_2)_{\partial E_2} : F  \rightarrow \partial E_2$ to be the composition $|f| \circ (\phi_1)_{\partial E_1}$. Similar to the argument above, the bordered 3-manifolds $(E_1, (\phi_1)_{\partial E_1})$ and $(E_2, (\phi_2)_{\partial E_2})$ are equivalent, which means that the associated type A structures $\widehat{\mathit{CFA}} (E_1, (\phi_1)_{\partial E_1})$ and $\widehat{\mathit{CFA}} (E_2, (\phi_2)_{\partial E_2})$ are homotopy equivalent. This implies $\widehat{\mathit{HFO}}(Y_1, K_1, n) \cong \widehat{\mathit{HFO}}(Y_2, K_2, n)$. 
\end{proof}

\begin{proof}[Proof that $\widehat{\mathit{HFO}}(Y^{\emph{orb}})$ generalizes $\widehat{\mathit{HF}}(Y)$]
When $n=1$, $N$ is modeled on $D^2 \times S^1$ and $D_N$ is the type D structure for $D^2 \times S^1$ with boundary parameterization given by $\alpha_1^a \mapsto \{1\} \times S^1$ and $\alpha_2^a \mapsto \partial D^2 \times \{1\}$. By \ref{pairing}, $ \widehat{\mathit{CF}} (Y^{\textrm{orb}}) \simeq\widehat{\mathit{CFA}} (E, \phi_{\partial E}) \boxtimes D_N = \widehat{\mathit{CFO}}(Y^{\textrm{orb}})$, which implies $ \widehat{\mathit{HF}}(Y^{\textrm{orb}}) \cong \widehat{\mathit{HFO}}(Y^{\textrm{orb}})$.
\end{proof}

\subsection{Examples}

In this subsection we calculate $\widehat{\mathit{HFO}}(Y^{\textrm{orb}})$ for some 3-orbifolds. For more examples see Section \ref{moreexs}.

\subsubsection{$(S^3, K, n)$, where $K$ is any knot in $S^3$} Given any choice of $N$ and $\phi: (D^2 \times S^1) / \mathbb{Z}_n  \rightarrow N$, we can represent $(E, \phi_{\partial E})$ by the bordered Heegaard diagram in Figure \ref{fig:S3a}. Then the associated type A structure $\widehat{\mathit{CFA}} (E, \phi_{\partial E})$ is given by the graph in Figure \ref{fig:S3b}. It's not hard to check that $\widehat{\mathit{CFA}} (E, \phi_{\partial E}) \boxtimes D_N'$ has trivial differential, which means $\widehat{\mathit{HFO}}(S^3, K, n) \cong \mathbb{Z}_2 \langle y \otimes x_1, \ldots, y \otimes x_n \rangle \cong \big(\mathbb{Z}_2 \big)^n$. Note that the rank of $\widehat{\mathit{HFO}}(S^3, K, n)$ is $n$ times the rank of $\widehat{\mathit{HF}}(S^3)$.

\begin{figure}[h]
    \centering
    \begin{subfigure}[t]{0.41\textwidth}
    \centering{
\resizebox{85mm}{!}{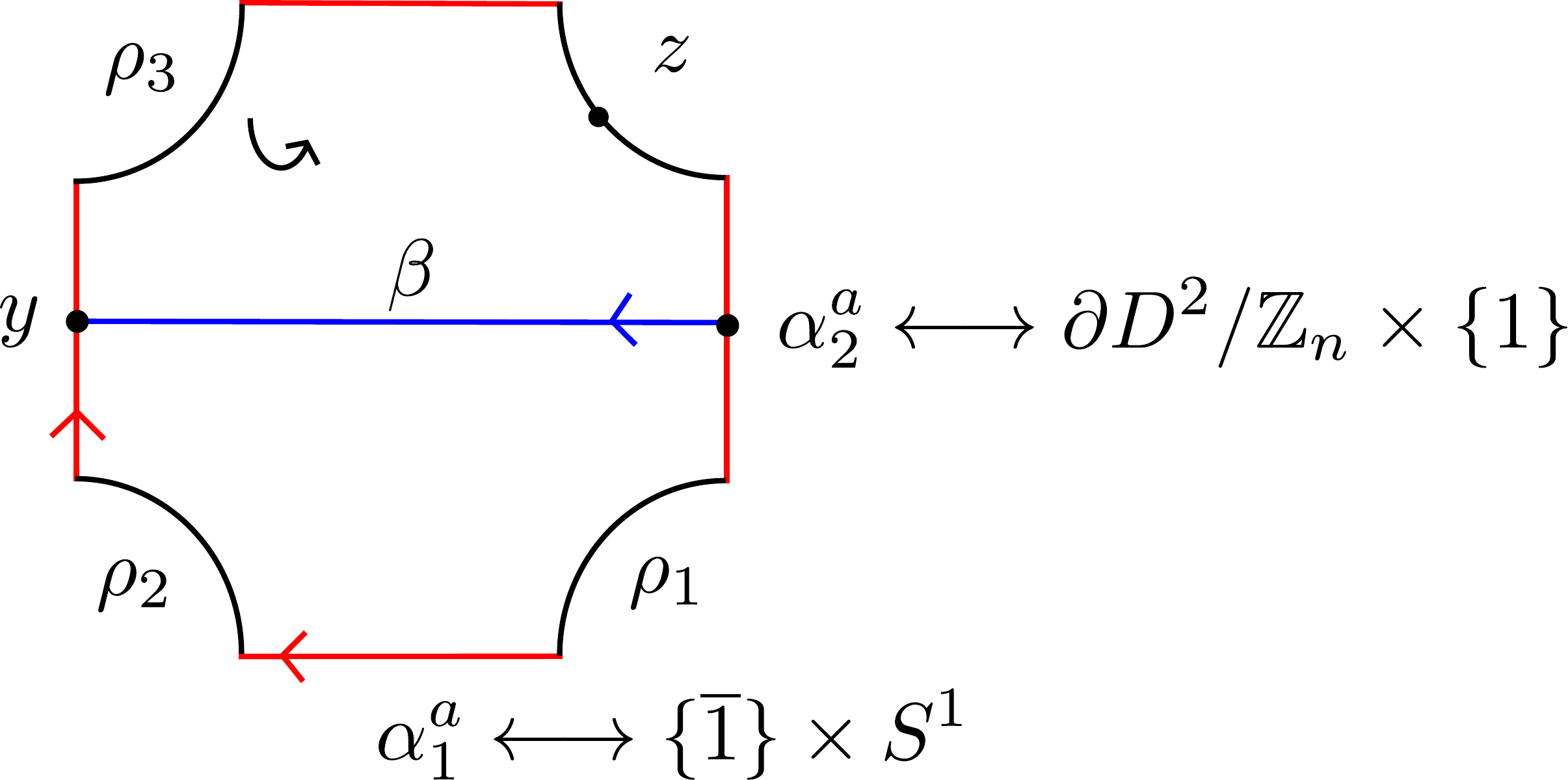}}
        \caption{}
        \label{fig:S3a}
    \end{subfigure}%
    ~ 
    \begin{subfigure}[t]{0.41\textwidth}
        \centering{
\resizebox{73mm}{!}{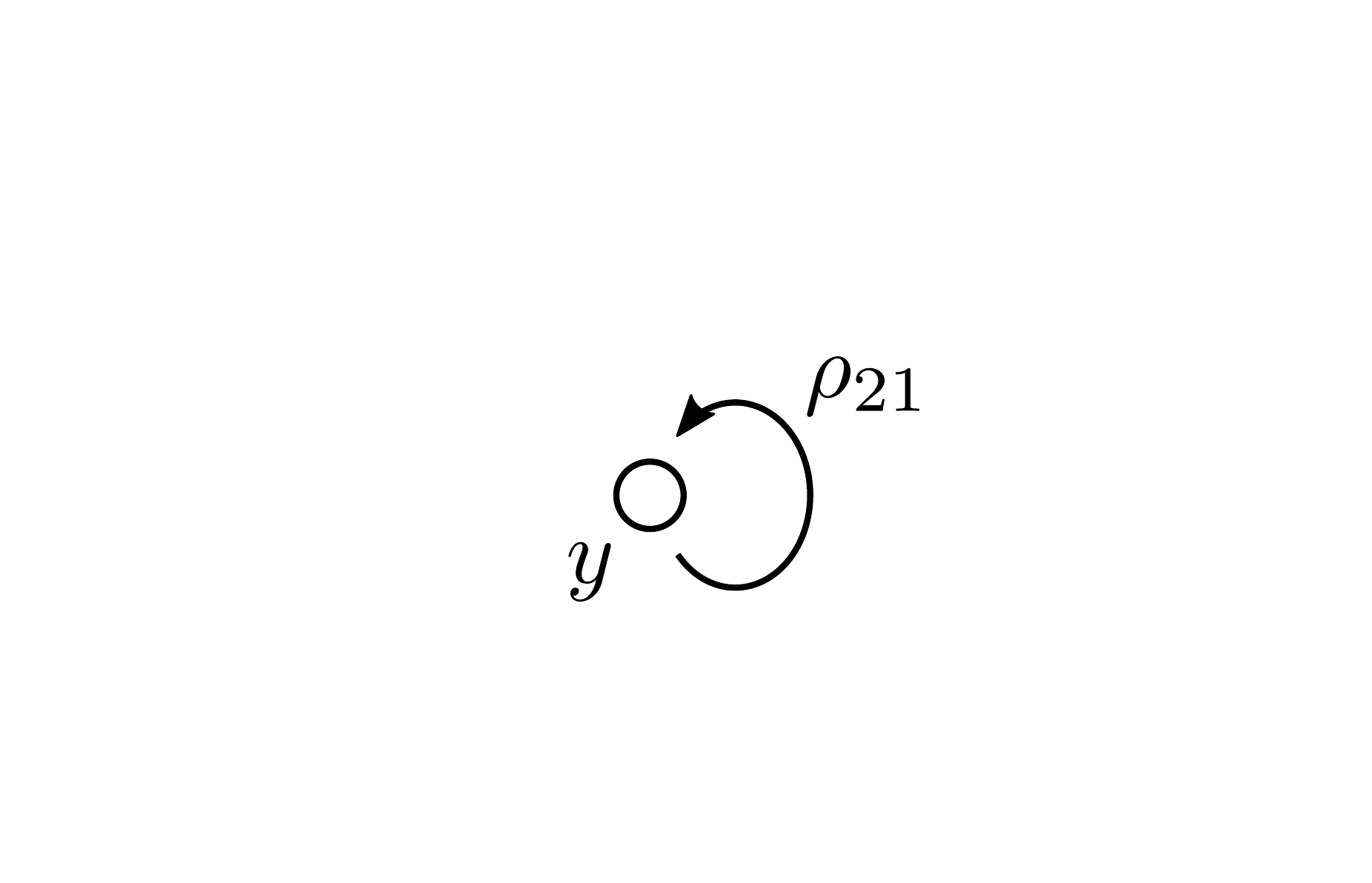}}
        \caption{}
        \label{fig:S3b}
    \end{subfigure}
\caption{On the left a bordered Heegaard diagram for $(E, \phi_{\partial E})$ when the 3-orbifold is $(S^3, K, n)$. On the right the corresponding type A structure $\widehat{\mathit{CFA}} (E, \phi_{\partial E})$.}
\end{figure}

\subsubsection{$(S^2 \times S^1, \{1\} \times S^1, n)$}\label{S2S1orbifoldex} In this example we take our bordered Heegaard diagram for $(E, \phi_{\partial E})$ to be Figure \ref{fig:S2xS1a}. The corresponding type A structure $\widehat{\mathit{CFA}} (E, \phi_{\partial E})$ is pictured in Figure \ref{fig:S2xS1b}. The differential in $\widehat{\mathit{CFA}} (E, \phi_{\partial E}) \boxtimes D_N'$ is again trivial, so $\widehat{\mathit{HFO}}(S^2 \times S^1, \{1\} \times S^1, n) \cong \mathbb{Z}_2 \langle y \otimes a, y \otimes b \rangle \cong \big(\mathbb{Z}_2 \big)^2$. Unlike the first example, the rank of $\widehat{\mathit{HFO}}(S^2 \times S^1, \{1\} \times S^1, n)$ equals the rank of $\widehat{\mathit{HF}}(S^2 \times S^1)$ for every $n$.

\begin{figure}[h]
    \centering
    \begin{subfigure}[t]{0.41\textwidth}
    \centering{
\resizebox{85mm}{!}{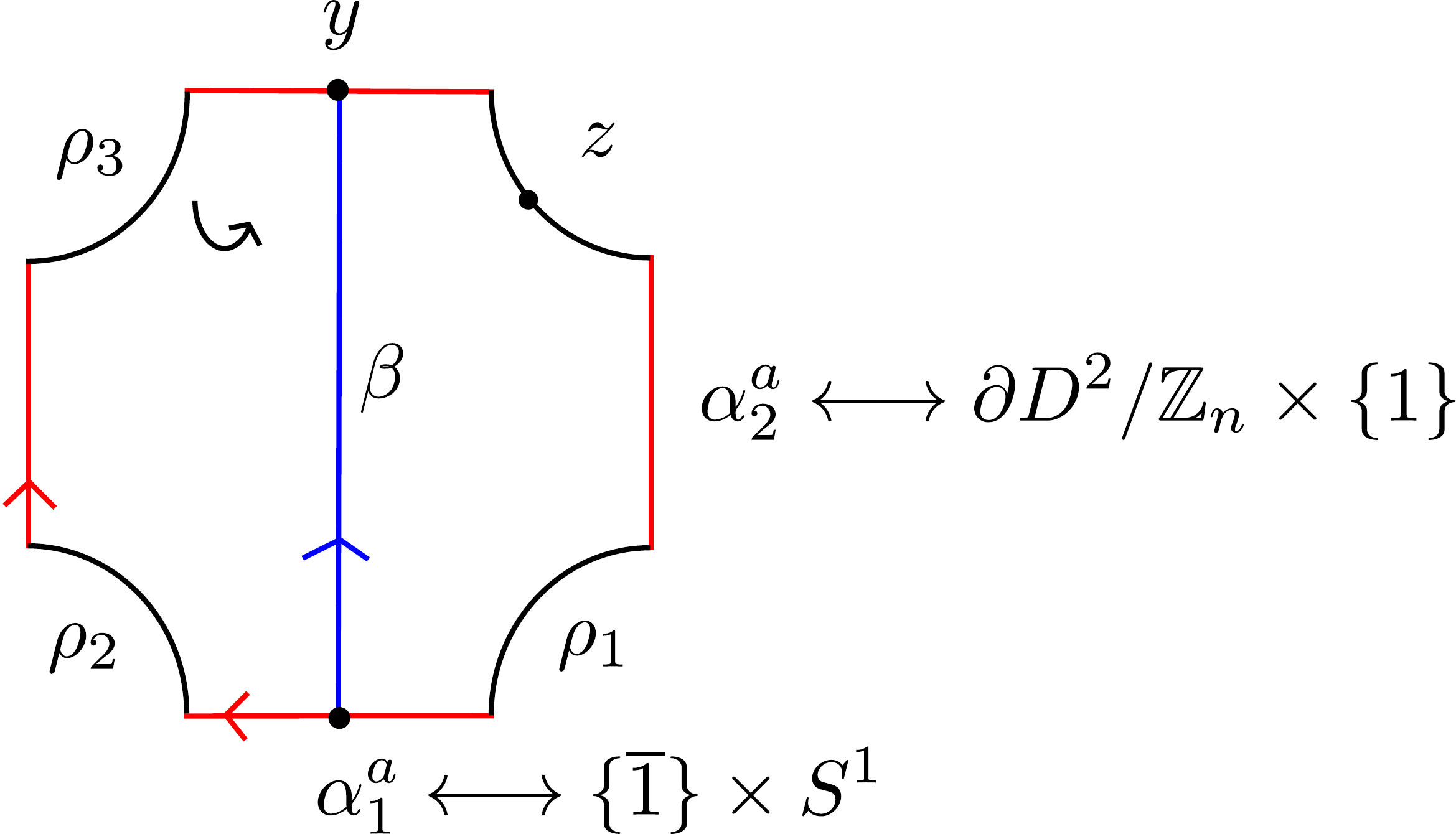}}
        \caption{}
        \label{fig:S2xS1a}
    \end{subfigure}%
    ~ 
    \begin{subfigure}[t]{0.41\textwidth}
        \centering{
\resizebox{73mm}{!}{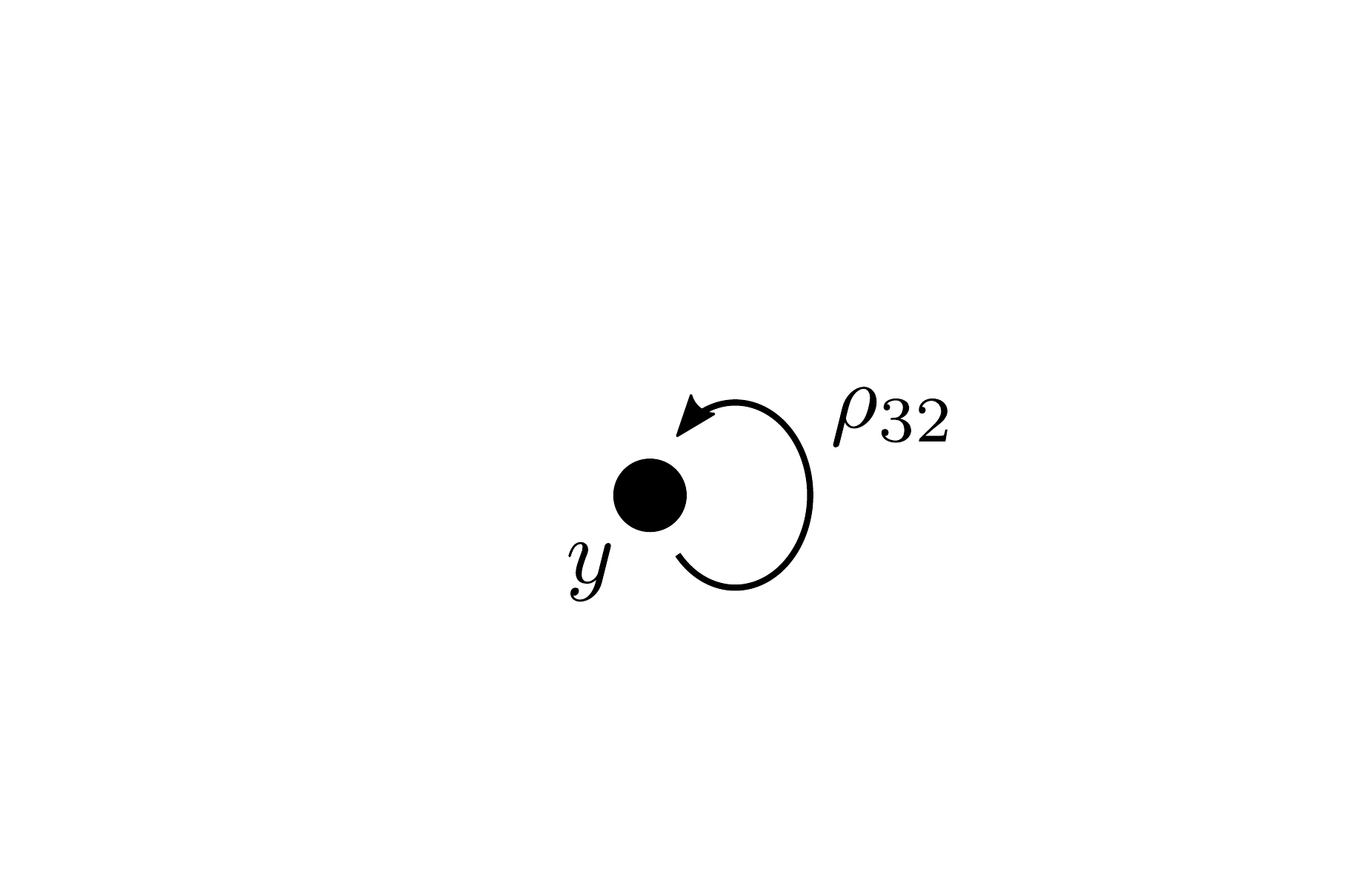}}
        \caption{}
        \label{fig:S2xS1b}
    \end{subfigure}
\caption{On the left a bordered Heegaard diagram for $(E, \phi_{\partial E})$ when the 3-orbifold is $(S^2 \times S^1, \{1\} \times S^1, n)$. On the right the corresponding type A structure $\widehat{\mathit{CFA}} (E, \phi_{\partial E})$.}
\end{figure}

\subsubsection{$\big(L(p, -q), K, n\big)$} Think of $L(p, -q)$ as two copies of $D^2 \times S^1$ glued together. Singularize one of them. This will be $N$ and $K$ will be the core of $N$. Take $p \geq 2$, $1 \leq q \leq p-1$, and $\textrm{gcd}(p,q)=1$. Then Figure \ref{fig:L(p,q)a} gives a bordered Heegaard diagram for $(E, \phi_{\partial E})$. The induced type A structure $\widehat{\mathit{CFA}} (E, \phi_{\partial E})$ is shown in Figure \ref{fig:L(p,q)b}. Note that $\widehat{\mathit{CFA}} (E, \phi_{\partial E})$ is bounded, unlike the previous examples. The differential in $\widehat{\mathit{CFO}}\big(L(p, -q), K, n\big) = \widehat{\mathit{CFA}} (E, \phi_{\partial E}) \boxtimes D_N$ is trivial, so $\widehat{\mathit{HFO}}\big(L(p, -q), K, n\big) = \mathbb{Z}_2 \langle y_1 \otimes x_1, \ldots, y_1 \otimes x_n, \ldots, y_p \otimes x_1, \ldots, y_p \otimes x_n \rangle \cong \big(\mathbb{Z}_2 \big)^{np}$. The rank of $\widehat{\mathit{HFO}}\big(L(p, -q), K, n\big)$ is $n$ times the rank of $\widehat{\mathit{HF}}\big(L(p, -q)\big)$.

\begin{figure}[h]
    \centering
    \begin{subfigure}[t]{0.6\textwidth}
    \centering{
\resizebox{90mm}{!}{\input{L_p_q_Example_1.pdf_tex}}}
        \caption{}
        \label{fig:L(p,q)a}
    \end{subfigure}%
    ~ 
    \begin{subfigure}[t]{0.4\textwidth}
        \centering{
\resizebox{60mm}{!}{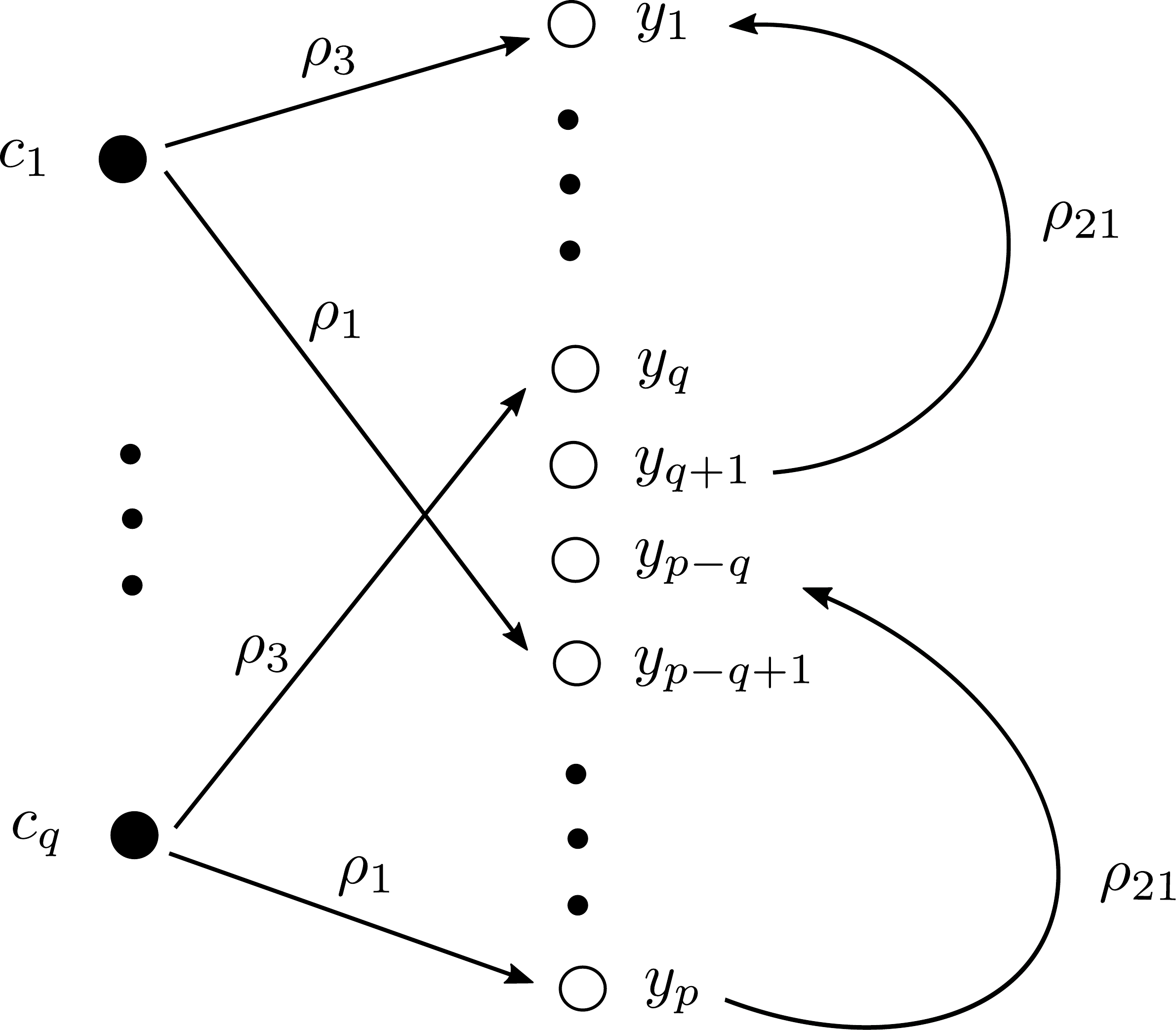}}
        \caption{}
        \label{fig:L(p,q)b}
    \end{subfigure}
\caption{On the left a bordered Heegaard diagram for $(E, \phi_{\partial E})$ when the 3-orbifold is $\big(L(p, -q), K, n\big)$. On the right the corresponding type A structure $\widehat{\mathit{CFA}} (E, \phi_{\partial E})$.}
\end{figure}

\section{Proof of Theorem \ref{relationshipHFhat}}\label{sec4}

We now restrict our attention to 3-orbifolds coming from integral surgeries on knots in $S^3$, and prove Theorem \ref{relationshipHFhat}. Let $Y$ be $r$-surgery on a knot $K \subset S^3$. Think of $Y$ as $(D^2 \times S^1) \cup E$, where $E$ is the exterior of $K$ in $S^3$ and we're identifying the meridian $\partial D^2 \times \{1\} \subset \partial D^2 \times S^1$ with the curve $\gamma = rm+\ell \subset \partial E$. If we replace $(D^2 \times S^1)$ with $(D^2 \times S^1) / \mathbb{Z}_n$, then we get the 3-orbifold $Y^{\textrm{orb}}=(Y, K, n)$. As in Section \ref{knotexteriors}, let $\phi_r: F \rightarrow \partial E$ be an orientation-preserving parameterization that sends $\alpha_1^a$ to $m$ and $\alpha_2^a$ to $\gamma$.

We first consider the case $\varepsilon (K) = 1$. By Part 1 of \cite[Lemma 3.2]{Hom}, we can find vertically and horizontally simplified bases $\{\bm{w_0}, \ldots, \bm{w_{2s}}\}$ and $\{\bm{w_0}', \ldots, \bm{w_{2s}}'\}$ for $\mathit{CFK}^-(K)$ over $\mathbb{Z}_2[U]$, with the following properties (possibly after reordering):

\begin{itemize}
\item $\bm{w_0}$ is the generator of the vertical complex $C^{\textrm{vert}}$,
\item $\bm{w_0}'$ is the generator of the horizontal complex $C^{\textrm{horz}}$,
\item $\partial ^{\textrm{horz}}\big(\overline{U^{A(\bm{w_1}')} \bm{w_1}'}\big)=\overline{U^{A(\bm{w_2}')} \bm{w_2}'}$, and
\item $\bm{w_2}' = \bm{w_0}$.
\end{itemize}

\noindent Fix such bases $\{\bm{w_0}, \ldots, \bm{w_{2s}}\}$ and $\{\bm{w_0}', \ldots, \bm{w_{2s}}'\}$ for $\mathit{CFK}^-(K)$. As discussed in Section \ref{knotexteriors}, any pair of horizontally and vertically simplified bases for $\mathit{CFK}^-(K)$ gives rise to a decorated, directed graph that represents $\widehat{\mathit{CFA}} (E, \phi_r)$. Let $\Gamma_r$ be the graph for $\widehat{\mathit{CFA}} (E, \phi_r)$ coming from $\{\bm{w_0}, \ldots, \bm{w_{2s}}\}$ and $\{\bm{w_0}', \ldots, \bm{w_{2s}}'\}$. We know that $\Gamma_r$ can't contain any coherently oriented cycles because $\bm{w_0} = \bm{w_2}' \neq \bm{w_0}'$ (and because there's no other way to get coherently oriented cycles in $\Gamma_r$). This implies that $\widehat{\mathit{CFA}} (E, \phi_r)$ is bounded.

Now consider the $\mathbb{Z}_2$-chain complexes
\[
\widehat{\mathit{CF}}(Y) = \Big( \widehat{\mathit{CFA}} (E, \phi_r) \boxtimes D_{D^2 \times S^1}, \partial^{\boxtimes} \Big)
\]
\noindent and
\[
\widehat{\mathit{CFO}}(Y^{\textrm{orb}}) = \Big( \widehat{\mathit{CFA}} (E, \phi_r) \boxtimes D_{(D^2 \times S^1) / \mathbb{Z}_n}, \partial^{\boxtimes}_{\textrm{orb}} \Big).
\]
\noindent We want to show $n \cdot \textrm{rank}\big(\widehat{\mathit{HF}}(Y)\big) = \textrm{rank}\big(\widehat{\mathit{HFO}}(Y^{\textrm{orb}})\big)$. We do this by comparing $\textrm{ker}(\partial^{\boxtimes})$ to $\textrm{ker}(\partial^{\boxtimes}_{\textrm{orb}})$, and $\textrm{im}(\partial^{\boxtimes})$ to $\textrm{im}(\partial^{\boxtimes}_{\textrm{orb}})$. Recall from Section \ref{knotexteriors} that $\bm{\kappa_e^i}, \bm{\lambda_f^i},$ and $\bm{\gamma_g}$ form a basis for $\widehat{\mathit{CFA}}(E, \phi_r) \cdot \iota_2$. Here $i \in \{0, \ldots, 2s\}$, $e \in \{1, \ldots, \ell_i\}$, $f \in \{1, \ldots, \ell_i'\}$, and $g \in \{1, \ldots, d=|2\tau(K)-r|\}$. For convenience, let $\bm{\alpha}$ be any one of these basis elements. Then $\widehat{\mathit{CF}}(Y)$ is generated by elements of the form $\bm{\alpha} \otimes \bm{x}$ and $\widehat{\mathit{CFO}}(Y^{\textrm{orb}})$ is generated by elements of the form $\bm{\alpha} \otimes \bm{x_j}$, where $j \in \{1, \ldots, n\}$.

\begin{Claim}\label{Claim1}
$n \cdot \emph{rank}\big(\emph{ker}(\partial^{\boxtimes})\big) = \emph{rank}\big(\emph{ker}(\partial^{\boxtimes}_{\emph{orb}})\big)$.
\end{Claim}

\begin{proof}
Because the type D structure map in $D_{(D^2 \times S^1) / \mathbb{Z}_n}$ is essentially $n$ copies of the type D structure map in $D_{D^2 \times S^1}$, $\partial^{\boxtimes} (\bm{\alpha} \otimes \bm{x})=0$ implies $\partial^{\boxtimes}_{\textrm{orb}} (\bm{\alpha} \otimes \bm{x_j})=0$ for every $j$. Since there is no other way for $\partial^{\boxtimes}_{\textrm{orb}}$ to be trivial on a basis element $\bm{\alpha} \otimes \bm{x_j}$ of $\widehat{\mathit{CFO}}(Y^{\textrm{orb}})$, we have that $n \cdot \textrm{rank}\big(\textrm{ker}(\partial^{\boxtimes})\big) = \textrm{rank}\big(\textrm{ker}(\partial^{\boxtimes}_{\textrm{orb}})\big)$.
\end{proof}

\begin{Claim}\label{Claim2}
Suppose $\partial^{\boxtimes} (\bm{\alpha} \otimes \bm{x}) \neq 0$. Then for every $j$, $\partial^{\boxtimes}_{\emph{orb}} (\bm{\alpha} \otimes \bm{x_j}) \neq 0$. Furthermore, there exists $\bm{\beta} \in \{\bm{\kappa_1^i}, \bm{\gamma_1}\}$ so that $\partial^{\boxtimes} (\bm{\alpha} \otimes \bm{x}) = \bm{\beta} \otimes \bm{x}$ and for every $j$, $\partial^{\boxtimes}_{\emph{orb}} (\bm{\alpha} \otimes \bm{x_j}) = \bm{\beta} \otimes \bm{x_{j+1}}$, with $j+1$ considered mod $n$.
\end{Claim}

\begin{proof}
The first statement is clear. As for the second one, if $\partial^{\boxtimes} (\bm{\alpha} \otimes \bm{x})$ is nontrivial, then $\bm{\beta}$ is the target of a directed edge labeled $\rho_3$ in $\Gamma_r$, and this happens exactly when $\bm{\beta} \in \{\bm{\kappa_1^i}, \bm{\gamma_1}\}$. For example, when $r < 2\tau(K)$, $\Gamma_r$ contains a piece that looks like
\[
\stackrel{\bm{\lambda_{\ell_1'}^1}}{\circ} \xrightarrow{\textrm{$\rho_2$}} \stackrel{\bm{w_2'=w_0}}{\bullet} \xrightarrow{\textrm{$\rho_{3}$}} \stackrel{\bm{\gamma_1}}{\circ}.    
\]
\noindent This gives us the nontrivial multiplication $m_2(\bm{\lambda_{\ell_1'}^1}, \rho_{23})= \bm{\gamma_1}$ in $\widehat{\mathit{CFA}} (E, \phi_r)$, which implies that $\partial^{\boxtimes} (\bm{\lambda_{\ell_1'}^1} \otimes \bm{x}) = \bm{\gamma_1} \otimes \bm{x}$ and $\partial^{\boxtimes}_{\textrm{orb}} (\bm{\lambda_{\ell_1'}^1} \otimes \bm{x_j}) = \bm{\gamma_1} \otimes \bm{x_{j+1}}$.
\end{proof}

\noindent It follows from Claims \ref{Claim1} and \ref{Claim2} that $n \cdot \textrm{rank}\big(\widehat{\mathit{HF}}(Y)\big) = \textrm{rank}\big(\widehat{\mathit{HFO}}(Y^{\textrm{orb}})\big)$.

A similar argument shows that when $\varepsilon(K)=-1$, $n \cdot \textrm{rank}\big(\widehat{\mathit{HF}}(Y)\big) = \textrm{rank}\big(\widehat{\mathit{HFO}}(Y^{\textrm{orb}})\big)$. This is because Part 2 of \cite[Lemma 3.2]{Hom} gives us vertically and horizontally simplified bases $\{\bm{w_0}, \ldots, \bm{w_{2s}}\}$ and $\{\bm{w_0}', \ldots, \bm{w_{2s}}'\}$ for $\mathit{CFK}^-(K)$ over $\mathbb{Z}_2[U]$, with the following properties (possibly after reordering):

\begin{itemize}
\item $\bm{w_0}$ is the generator of the vertical complex $C^{\textrm{vert}}$,
\item $\bm{w_0}'$ is the generator of the horizontal complex $C^{\textrm{horz}}$,
\item $\partial ^{\textrm{horz}}\big(\overline{U^{A(\bm{w_1}')} \bm{w_1}'}\big)=\overline{U^{A(\bm{w_2}')} \bm{w_2}'}$, and
\item $\bm{w_1}' = \bm{w_0}$.
\end{itemize}

Now suppose $\varepsilon(K)=0$. By \cite[Lemma 3.3]{Hom}, we can find vertically and horizontally simplified bases $\{\bm{w_0}, \ldots, \bm{w_{2s}}\}$ and $\{\bm{w_0}', \ldots, \bm{w_{2s}}'\}$ for $\mathit{CFK}^-(K)$ over $\mathbb{Z}_2[U]$ so that the generator $\bm{w_0}$ of the vertical complex $C^{\textrm{vert}}$ equals the generator $\bm{w_0}'$ of the horizontal complex $C^{\textrm{horz}}$. Fix such bases $\{\bm{w_0}, \ldots, \bm{w_{2s}}\}$ and $\{\bm{w_0}', \ldots, \bm{w_{2s}}'\}$. Let $\Gamma_r$ be the graph for $\widehat{\mathit{CFA}} (E, \phi_r)$ coming from $\{\bm{w_0}, \ldots, \bm{w_{2s}}\}$ and $\{\bm{w_0}', \ldots, \bm{w_{2s}}'\}$. Note that $\tau(K)=0$ because $\varepsilon(K)=0$. We have two cases: either $r \neq 0$ or $r=0$. If $r \neq 0$, then $r \neq 2\tau(K)$. This means that  $\Gamma_r$ doesn't contain any coherently oriented cycles, and we can use the argument in the $\varepsilon(K)=1$ case above to show that $n \cdot \textrm{rank}\big(\widehat{\mathit{HF}}(Y)\big) = \textrm{rank}\big(\widehat{\mathit{HFO}}(Y^{\textrm{orb}})\big)$.

Assume $r=0$. Then $r=2\tau(K)$ and the unstable chain in $\Gamma_r$ is a coherently oriented cycle, which implies that $\widehat{\mathit{CFA}} (E, \phi_r)$ is unbounded. Since the unstable chain doesn't interact with the rest of the type A structure, we can express $\widehat{\mathit{CFA}} (E, \phi_r)$ as $\widehat{\mathit{CFA}} (E, \phi_r)_1 \oplus \widehat{\mathit{CFA}} (E, \phi_r)_2$, where $\widehat{\mathit{CFA}} (E, \phi_r)_1$ is the unbounded type A structure corresponding to the unstable chain and $\widehat{\mathit{CFA}} (E, \phi_r)_2$ is the bounded type A structure corresponding to the complement of the unstable chain. Then we have
\[
\widehat{\mathit{CF}}(Y) \simeq \Big( \widehat{\mathit{CFA}}(E, \phi_r)_1 \boxtimes D_{D^2 \times S^1}' \Big) \oplus \Big( \widehat{\mathit{CFA}}(E, \phi_r)_2 \boxtimes D_{D^2 \times S^1} \Big)
\]
\noindent and
\[
\widehat{\mathit{CFO}}(Y^{\textrm{orb}}) \simeq \Big( \widehat{\mathit{CFA}}(E, \phi_r)_1 \boxtimes D_{(D^2 \times S^1) / \mathbb{Z}_n}' \Big) \oplus \Big( \widehat{\mathit{CFA}}(E, \phi_r)_2 \boxtimes D_{(D^2 \times S^1) / \mathbb{Z}_n} \Big),
\]
\noindent where $D_{D^2 \times S^1}'$ and $D_{(D^2 \times S^1) / \mathbb{Z}_n}'$ are the bounded type D structures in Figure \ref{fig:Bounded_DN}.

This means that $\widehat{\mathit{HF}}(Y)$ and $\widehat{\mathit{HFO}}(Y^{\textrm{orb}})$ admit the following decompositions:
\[
\widehat{\mathit{HF}}(Y) \cong H_1 \oplus H_2
\]
\noindent and
\[
\widehat{\mathit{HFO}}(Y^{\textrm{orb}}) \cong H_1^{\textrm{orb}} \oplus H_2^{\textrm{orb}},
\]
\noindent where $H_i$ denotes the homology of the $i$th piece in $\widehat{\mathit{CF}}(Y)$ and $H_i^{\textrm{orb}}$ denotes the homology of the $i$th piece in $\widehat{\mathit{CFO}}(Y^{\textrm{orb}})$. From Example \ref{S2S1orbifoldex}, we have that $\textrm{rank}(H_1^{\textrm{orb}}) = 2 = \textrm{rank}(H_1)$. By the argument in the $\varepsilon(K)=1$ case, $\textrm{rank}(H_2^{\textrm{orb}}) = n \cdot \textrm{rank}(H_2)$. Consequently, we get that $\textrm{rank}\big(\widehat{\mathit{HFO}}(Y^{\textrm{orb}}) \big) = 2 + n \cdot \textrm{rank}(H_2)$, which implies that $\textrm{rank}\big(\widehat{\mathit{HFO}}(Y^{\textrm{orb}}) \big) = n \cdot \textrm{rank}\big(\widehat{\mathit{HF}}(Y)\big) -2n +2$, as needed. 

{\qed} \popQED

\subsection{Examples}\label{moreexs}

We conclude with a couple of examples.

\subsubsection{}Let $K$ be the left-handed trefoil $T(2,-3)$. Fix $n \in \mathbb{Z}$. Take $r=0$. Then $\widehat{\mathit{CFA}}(E, \phi_0)$ is given by the graph in Figure \ref{fig:LHT}, and $\widehat{\mathit{CFO}}(Y^{\textrm{orb}}) = \widehat{\mathit{CFA}}(E, \phi_0) \boxtimes D_{(D^2 \times S^1) / \mathbb{Z}_n}$ is generated by $\bm{\kappa^1_1} \otimes \bm{x_j}$, $\bm{\lambda^1_1} \otimes \bm{x_j}$, $\bm{\gamma_1} \otimes \bm{x_j}$, and $\bm{\gamma_2} \otimes \bm{x_j}$. The only nontrivial differential is $\partial^{\boxtimes}(\bm{\gamma_2} \otimes \bm{x_j}) = \bm{\kappa^1_1} \otimes \bm{x_{j+1}}$. This implies that $\widehat{\mathit{HFO}}(Y^\textrm{orb}) = \mathbb{Z}_2 \langle \bm{\lambda^1_1} \otimes \bm{x_j}, \bm{\gamma_1} \otimes \bm{x_j} \rangle$, which has rank $2n$. Note that this agrees with Theorem \ref{relationshipHFhat}, since $\varepsilon(K)=-1$ and $\textrm{rank}\big( \widehat{\mathit{HF}}(Y) \big) =2$.

\begin{figure}[h]
\centering{
\resizebox{65mm}{!}{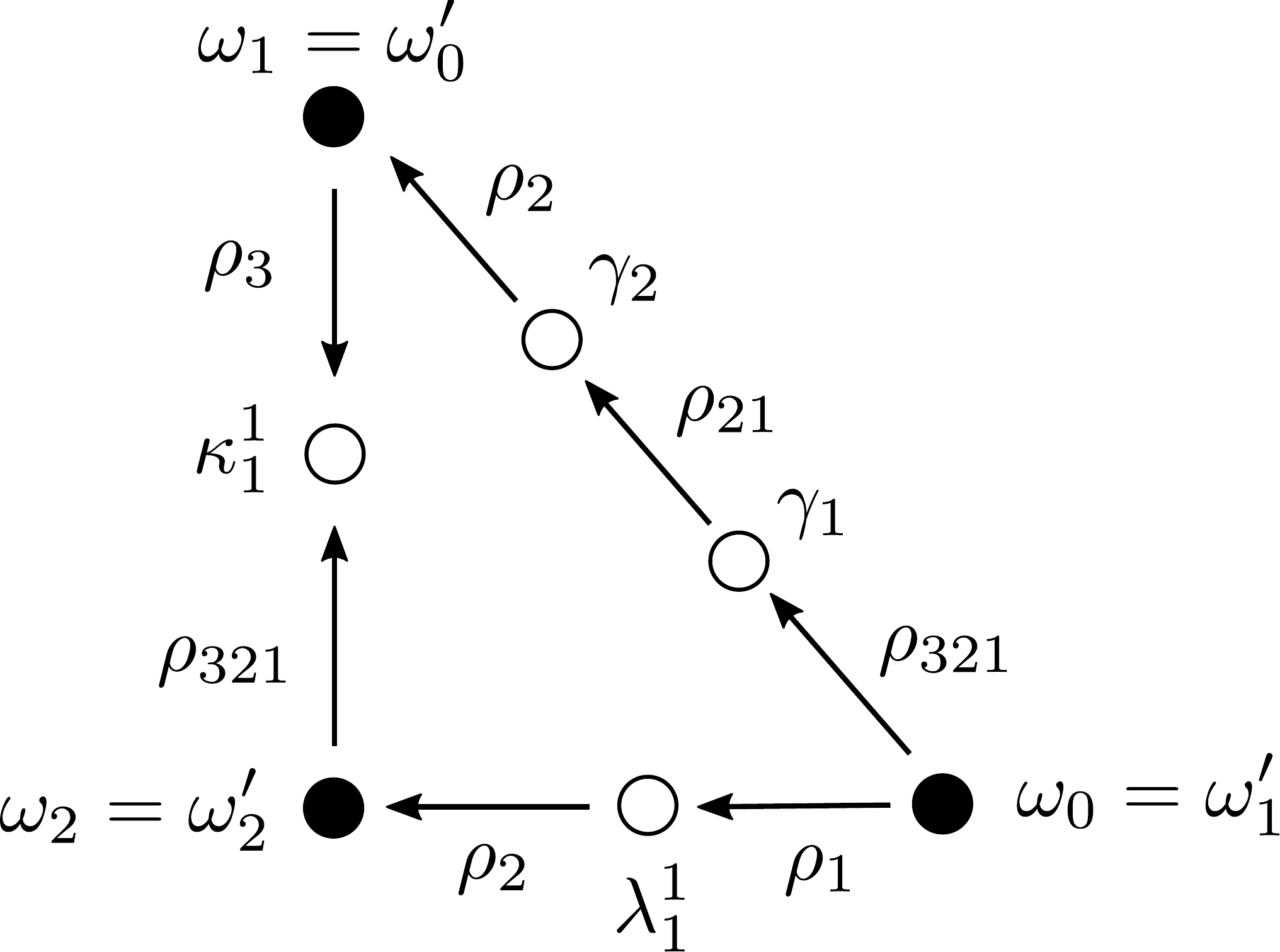}
\caption{$\widehat{\mathit{CFA}}(E, \phi_0)$}
\label{fig:LHT}
}
\end{figure}

\subsubsection{}Let $K$ be the figure-eight knot. Again fix $n \in \mathbb{Z}$ and assume $r=0$. $\widehat{\mathit{CFA}}(E, \phi_0)$ is given by the graph in Figure \ref{fig:Figure8}. Let $C_1$ denote the unbounded type A structure represented by the unstable loop, and let $C_2$ be the bounded type A structure represented by everything else. Then 
\[
\widehat{\mathit{CFO}}(Y^{\textrm{orb}}) \simeq \Big( C_1 \boxtimes D_{(D^2 \times S^1) / \mathbb{Z}_n}' \Big) \oplus \Big( C_2 \boxtimes D_{(D^2 \times S^1) / \mathbb{Z}_n} \Big),
\]
\noindent which implies that
\[
\widehat{\mathit{HFO}}(Y^{\textrm{orb}}) \cong H_1^{\textrm{orb}} \oplus H_2^{\textrm{orb}},
\]
\noindent where $H_1^{\textrm{orb}}$ denotes the homology of $C_1 \boxtimes D_{(D^2 \times S^1) / \mathbb{Z}_n}'$ and $H_2^{\textrm{orb}}$ denotes the homology of $C_2 \boxtimes D_{(D^2 \times S^1) / \mathbb{Z}_n}$. As noted above, Example \ref{S2S1orbifoldex} tells us that $H_1^{\textrm{orb}} \cong \langle \bm{\omega_0} \otimes \bm{a}, \bm{\omega_0} \otimes \bm{b} \rangle$. Now $C_2 \boxtimes D_{(D^2 \times S^1) / \mathbb{Z}_n}$ is generated by $\bm{\lambda^1_1} \otimes \bm{x_j}$, $\bm{\lambda^3_1} \otimes \bm{x_j}$, $\bm{\kappa^1_1} \otimes \bm{x_j}$, and $\bm{\kappa^3_1} \otimes \bm{x_j}$. Let $\partial_2^{\boxtimes}$ denote the differential in $C_2 \boxtimes D_{(D^2 \times S^1) / \mathbb{Z}_n}$. Then $\partial_2^{\boxtimes}(\bm{\lambda^1_1} \otimes \bm{x_j}) = \bm{\kappa^3_1} \otimes \bm{x_{j+1}}$, and on all other generators $\partial_2^{\boxtimes}$ is trivial. This implies that $H_2^{\textrm{orb}} = \mathbb{Z}_2 \langle \bm{\lambda^3_1} \otimes \bm{x_j}, \bm{\kappa^1_1} \otimes \bm{x_j} \rangle$, which has rank $2n$. Altogether, $\widehat{\mathit{HFO}}(Y^{\textrm{orb}})$ has rank $2n+2$, which agrees with Theorem \ref{relationshipHFhat}, since $\varepsilon(K)=0$ and $\textrm{rank}\big( \widehat{\mathit{HF}}(Y) \big) =4$.

\begin{figure}[h]
\centering{
\resizebox{85mm}{!}{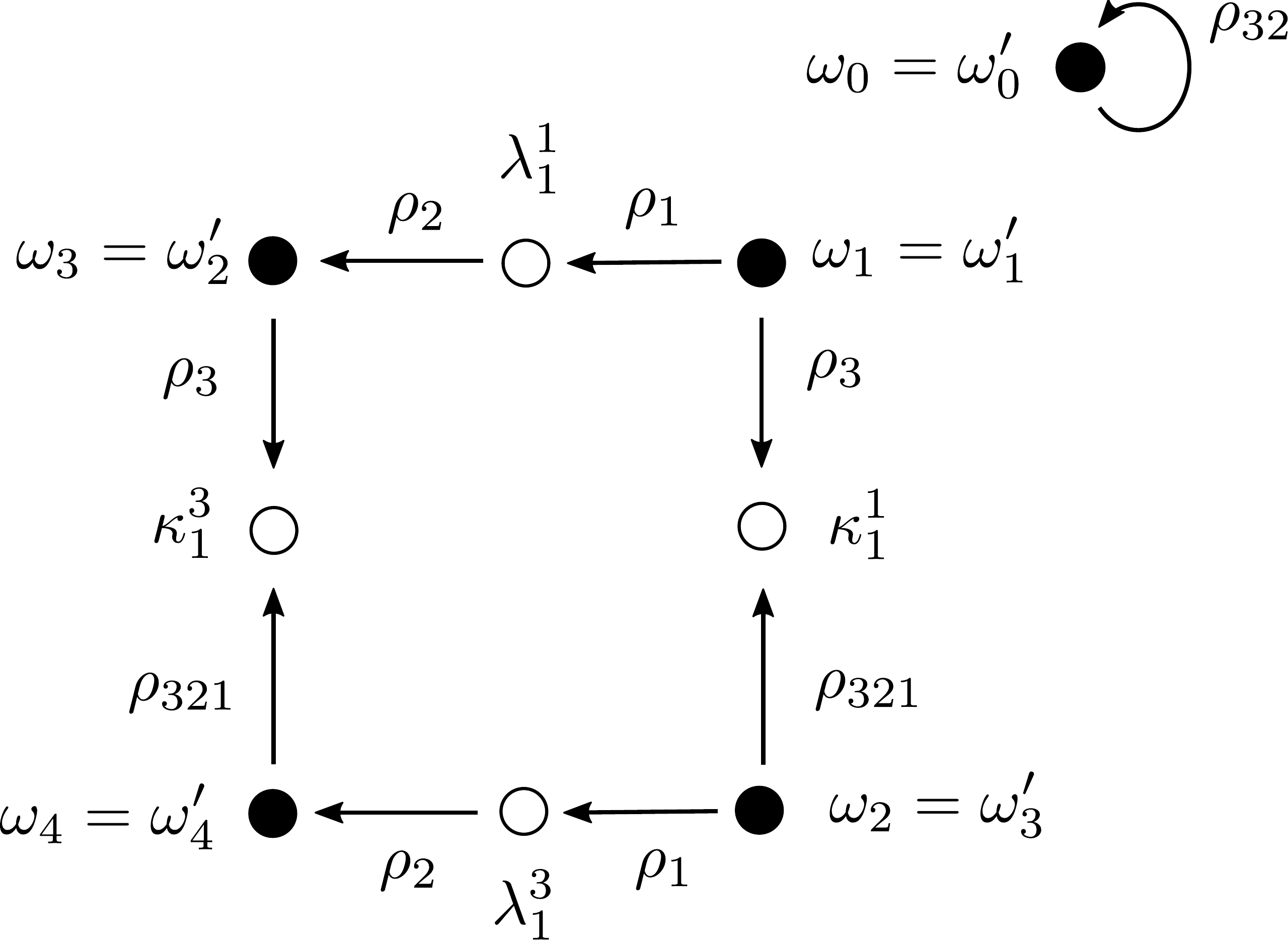}
\caption{$\widehat{\mathit{CFA}}(E, \phi_0)$}
\label{fig:Figure8}
}
\end{figure}

\section{Categorifying $|H_1^{\textrm{orb}}|$ to $\widehat{\mathit{HFO}}(Y^{\textrm{orb}})$}\label{sec5}

\subsection{Background}\label{relativeZ2grading}

We start by reviewing the relative $\mathbb{Z}_2$-grading $\textrm{gr}$ on $\widehat{\mathit{CF}}$. The details are in \cite{OS2, HLW}. Let $\mathcal{H} = (\Sigma_g, \bm{\alpha}, \bm{\beta}, z)$ be a Heegaard diagram for a closed 3-manifold $Y$. Order and orient the $\bm{\alpha}$ and $\bm{\beta}$ circles. Then given any generator $\bm{x}$ of the $\mathbb{Z}_2$-chain complex $\widehat{\mathit{CF}}(\mathcal{H})$, we have two integers $\textrm{inv}(\sigma_{\bm{x}})$ and $o(\bm{x})$ defined as follows. $\sigma_{\bm{x}}$ is the permutation in $S_g$ that allows us to express $\bm{x}$ as $\{x_1, \ldots, x_g\}$ where $x_i \in \beta_i \cap \alpha_{\sigma_{\bm{x}}(i)}$, and $\textrm{inv}(\sigma_{\bm{x}})$ counts the number of inversions in $\sigma_{\bm{x}}$, i.e. the number of pairs $(i,j)$ where $i < j$, but $\sigma_{\bm{x}}(i) > \sigma_{\bm{x}}(j)$. At every intersection point $x_i$ we can assign an orientation: positive if $\alpha_{\sigma_{\bm{x}}(i)}$ followed by $\beta_i$ gives the orientation on $\Sigma_g$, and negative otherwise. Write $o(x_i)=0$ if $x_i$ is positively oriented and $o(x_i)=1$ if $x_i$ is negative oriented. Then $o(\bm{x})$ is the sum $o(x_1) + \ldots + o(x_g)$, and we define
\[
\textrm{gr}(\bm{x}) = \textrm{inv}(\sigma_{\bm{x}}) + o(\bm{x}) \textrm{   (mod 2)}.
\]
\noindent Up to a possible overall shift, $\textrm{gr}$ is well-defined, i.e. does not depend on how we order and orient the $\bm{\alpha}$ and $\bm{\beta}$ circles. So we'll think of $\textrm{gr}$ as a relative $\mathbb{Z}_2$-grading on $\widehat{\mathit{CF}}(\mathcal{H})$. $\textrm{gr}$ induces a relative $\mathbb{Z}_2$-grading on $\widehat{\mathit{HF}}(\mathcal{H})$, which we also call $\textrm{gr}$. With respect to both relative $\mathbb{Z}_2$-gradings, we have $|\chi\big(\widehat{\mathit{CF}}(\mathcal{H})\big)| = |\chi\big(\widehat{\mathit{HF}}(\mathcal{H})\big)| = |H_1(Y)|$, for details see \cite{OS2}.

There's an analogous story for the bordered invariants $\widehat{\mathit{CFA}}$ and $\widehat{\mathit{CFD}}$, due to Hom, Lidman, and Watson in \cite{HLW}. To explain this, we'll need the notion of a bordered partial permutation. Recall $[n]$ denotes the set $\{1, \ldots, n\}$.

\begin{Def}
Let $ g \in \mathbb{N}$. Fix $B \subseteq [g+1]$ with $|B|=2$. Suppose $\sigma: [g] \rightarrow [g+1]$ is a function that satisfies the following:
\begin{enumerate}
\item $\sigma$ is injective and
\item the complement of $B$ in $[g+1]$ lies in $\textrm{Im}(\sigma)$.
\end{enumerate}
Then we call $\sigma$ a \textit{bordered partial permutation}. Furthermore, we say that $\sigma$ is \textit{type A} if $B=\{g, g+1\}$, and \textit{type D} if $B=\{1, 2\}$.
\end{Def}

\noindent Given a bordered partial permutation $\sigma$, we can consider its \textit{sign} $\textrm{sgn}(\sigma)$. For type A bordered partial permutations $\sigma$, we define $\textrm{sgn}_A(\sigma)= \textrm{inv}(\sigma) \textrm{ (mod 2)}$, and for type D bordered partial permutations $\sigma$, we define $\textrm{sgn}_D(\sigma)= \textrm{inv}(\sigma) + \sum_{i \in \textrm{Im}(\sigma)} \# \{ j \mid j >i, j \notin \textrm{Im}(\sigma)\}$ \textrm{(mod 2)}.

Now let $\mathcal{H} = (\overline{\Sigma}_g; \bm{\alpha}; \bm{\beta}; z)$ be a bordered Heegaard diagram for a bordered 3-manifold $(Y, \phi)$. There's a canonical way to order and orient the two $\bm{\alpha}$ arcs $\bm{\alpha}_1$ and $\bm{\alpha}_2$. The ordering is given by the indices. The orientations are defined as follows. If $(Y, \phi)$ is type A, we orient $\bm{\alpha}^a_1$ and $\bm{\alpha}^a_2$ so that when we follow $\partial \overline{\Sigma}_g$ in the direction of its orientation, we hit the initial point of $\bm{\alpha}^a_1$, then the initial point of $\bm{\alpha}^a_2$, followed by the terminal point of $\bm{\alpha}^a_1$ and then the terminal point of $\bm{\alpha}^a_2$. If $(Y, \phi)$ is type D, we orient $\bm{\alpha}^a_1$ and $\bm{\alpha}^a_2$ so that when we follow $\partial \overline{\Sigma}_g$ in the direction of its orientation, we hit the initial point of $\bm{\alpha}^a_2$, then the initial point of $\bm{\alpha}^a_1$, followed by the terminal point of $\bm{\alpha}^a_2$ and then the terminal point of $\bm{\alpha}^a_1$. Doing this ensures that when we glue the type A and type D $\bm{\alpha}^a_i$ arcs together along their boundaries, we get a coherently oriented $\bm{\alpha}^a_i$ circle. Now fix an ordering of the $\bm{\alpha}$ and $\bm{\beta}$ circles. If $(Y, \phi)$ is type A, we assume the circles in $\bm{\alpha}^c$ are ordered before the arcs in $\bm{\alpha}^a$. If $(Y, \phi)$ is type D, we choose the opposite ordering: $\bm{\alpha}^a$ before $\bm{\alpha}^c$. This, coupled with the above ordering on the $\bm{\alpha}$ arcs, is an ordering on all of $\bm{\alpha}$ and $\bm{\beta}$. Note that if we fix orientations on the $\bm{\alpha}$ and $\bm{\beta}$ circles, then we've oriented all of $\bm{\alpha}$ and $\bm{\beta}$.

Let $\bm{x}$ be a $g$-tuple of points in $\bm{\beta} \cap \bm{\alpha}$, with one point on each $\bm{\beta}$ circle, one point on each of the $g-1$ $\bm{\alpha}^c$ circles, and one point on one of the two $\bm{\alpha}^a$ arcs. Express $\bm{x}$ as $\{x_1, \ldots, x_g\}$, where $x_i \in \beta_i \cap \alpha_{\sigma_{\bm{x}}(i)}$ for some injection $\sigma_{\bm{x}}: [g] \rightarrow [g+1]$ satisfying $\bm{\alpha}^c \subset \{\alpha_{\sigma_{\bm{x}}(i)} \mid i \in [g]\}$. When $(Y, \phi)$ is type A, $\sigma_{\bm{x}}$ is a type A bordered partial permutation, and when $(Y, \phi)$ is type D, $\sigma_{\bm{x}}$ is a type D bordered partial permutation. We can now define the relative $\mathbb{Z}_2$-gradings $\textrm{gr}_A$ and $\textrm{gr}_D$ on $\widehat{\mathit{CFA}}$ and $\widehat{\mathit{CFD}}$:

\begin{Def}
The \textit{type A grading} of a generator $\bm{x}$ of $\widehat{\mathit{CFA}}(\mathcal{H})$ is 
\[
\textrm{gr}_A(\bm{x}) = \textrm{sgn}_A(\sigma_{\bm{x}}) + o(\bm{x}) \textrm{ (mod 2).}
\]
\end{Def}

\begin{Def}
The \textit{type D grading} of a generator $\bm{x}$ of $\widehat{\mathit{CFD}}(\mathcal{H})$ is
\[
\textrm{gr}_D(\bm{x}) = \textrm{sgn}_D(\sigma_{\bm{x}}) + o(\bm{x}) \textrm{ (mod 2).}
\]
\end{Def}

\noindent Up to a possible overall shift, $\textrm{gr}_A$ and $\textrm{gr}_D$ do not depend on how we order and orient the $\bm{\alpha}$ and $\bm{\beta}$ circles. So we'll think of $\textrm{gr}_A$ and $\textrm{gr}_D$ as relative $\mathbb{Z}_2$-gradings on $\widehat{\mathit{CFA}}$ and $\widehat{\mathit{CFD}}$.

\begin{Ex}\label{regulargradingsolidtorus}
Consider $D^2 \times S^1$ with the type D parameterization $\psi: F \rightarrow \partial (D^2 \times S^1)$ defined by $\alpha_1^a \mapsto \{1\}\times S^1$ and $\alpha_2^a \mapsto \partial D^2 \times \{1\}$. Let $\mathcal{H}_{D^2 \times S^1}$ be the bordered Heegaard diagram for $(D^2 \times S^1, \psi)$ in Figure \ref{fig:BHD}. Then the type D grading $\textrm{gr}_D$ on $\widehat{\mathit{CFD}}(\mathcal{H}_{D^2 \times S^1})$ is given by $\bm{x} \mapsto 1$. Note that if we change the orientation on $\bm{\beta}$, we get $\textrm{gr}_D(\bm{x})=0$ instead.
\end{Ex}

We next explain how to recover the relative $\mathbb{Z}_2$-grading $\textrm{gr}$ on $\widehat{\mathit{CF}}$ from the relative type A and type D gradings $\textrm{gr}_A$ and $\textrm{gr}_D$ on $\widehat{\mathit{CFA}}$ and $\widehat{\mathit{CFD}}$. This is due to Hom, Lidman, and Watson in \cite[Proposition 3.17]{HLW}. Let $(\mathcal{H}_1, \mathcal{Z})$ and $(\mathcal{H}_2, -\mathcal{Z})$ be bordered Heegaard diagrams for $(Y_1, F)$ and $(Y_2, -F)$. If we glue $(\mathcal{H}_1, \mathcal{Z})$ and $(\mathcal{H}_2, -\mathcal{Z})$ together along $\mathcal{Z}$, we get a Heegaard diagram $\mathcal{H}=\mathcal{H}_1 \cup_{\mathcal{Z}} \mathcal{H}_2$ that describes the closed 3-manifold $Y_1 \cup_{F} Y_2$. In particular, the $\bm{\alpha}$ arcs in $(\mathcal{H}_1, \mathcal{Z})$ and $(\mathcal{H}_2, -\mathcal{Z})$ give rise to two $\bm{\alpha}$ circles in $\mathcal{H}$, and the preferred orientations on the $\bm{\alpha}$ arcs induce coherent orientations on the resulting $\bm{\alpha}$ circles. Furthermore, if we orient the $\bm{\alpha}$ and $\bm{\beta}$ circles in $(\mathcal{H}_1, \mathcal{Z})$ and $(\mathcal{H}_2, -\mathcal{Z})$, we get induced orientations for the remaining $\bm{\alpha}$ and $\bm{\beta}$ circles in $\mathcal{H}$. In a similar way, given any ordering on the $\bm{\alpha}$ and $\bm{\beta}$ circles in $(\mathcal{H}_1, \mathcal{Z})$ and $(\mathcal{H}_2, -\mathcal{Z})$, there is an induced ordering on the $\bm{\alpha}$ and $\bm{\beta}$ circles in $\mathcal{H}$. To get the ordering on the $\bm{\alpha}$ circles in $\mathcal{H}$, we take the $\bm{\alpha}$ circles in $(\mathcal{H}_1, \mathcal{Z})$ first, followed by the glued up $\bm{\alpha}$ arcs in $\mathcal{H}$, and then the $\bm{\alpha}$ circles in $(\mathcal{H}_2, -\mathcal{Z})$. The ordering on the $\bm{\beta}$ circles in $\mathcal{H}$ is similar. 

Now let $\bm{y}$ and $\bm{x}$ be generators of $\widehat{\mathit{CFA}}(\mathcal{H}_1, \mathcal{Z})$ and $\widehat{\mathit{CFD}}(\mathcal{H}_2, -\mathcal{Z})$, respectively. Suppose $\bm{y} \otimes \bm{x} \neq 0$. Then $\bm{y} \otimes \bm{x}$ is a generator of $\widehat{\mathit{CF}}(\mathcal{H})$, and \cite[Proposition 3.17]{HLW} states that up to a possible overall shift independent of both $\bm{y}$ and $\bm{x}$
\begin{equation}\label{Eqn51}
\textrm{gr}(\bm{y} \otimes \bm{x}) = \textrm{gr}_A(\bm{y}) + \textrm{gr}_D(\bm{x})  \textrm{ (mod 2).}
\end{equation}

\subsection{Proof of Theorem \ref{categorification}} 

Recall we have the following set-up: $Y^{\textrm{orb}}$ is a 3-orbifold with singular set a knot $K$ of multiplicity $n$, $N$ is a $\mathbb{Z}_n$-equivariant tubular neighborhood of $K$ parameterized by $\phi_N: (D^2 \times S^1) / \mathbb{Z}_n  \rightarrow N$, and $E$ is the complement of $\textrm{int}(N)$ with (orientation-preserving) boundary parameterization $\phi_{\partial E}: F \rightarrow \partial E$ induced by $\phi_N$. Choose a bordered Heegaard diagram $(\mathcal{H}_E, \mathcal{Z})$ for the type A bordered 3-manifold $(E, \phi_{\partial E})$. Without loss of generality, we'll assume the associated type A structure $\widehat{\mathit{CFA}}(\mathcal{H}_E, \mathcal{Z})$ is bounded. Let $\textrm{gr}_A$ be the relative $\mathbb{Z}_2$-grading on $\widehat{\mathit{CFA}}(\mathcal{H}_E, \mathcal{Z})$ coming from bordered Floer theory. Figure \ref{fig:OrbifoldBHD} gives an orbifold bordered Heegaard diagram for $N$; call this  $(\mathcal{H}_N, -\mathcal{Z})$. We can define a relative $\mathbb{Z}_2$-grading $\textrm{gr}_D^{\textrm{orb}}$ on the type D structure $D_N$ by setting $\textrm{gr}_D^{\textrm{orb}}(\bm{x_i})=1$ for every $i$. If we pick a different orientation on $\bm{\beta}$, we'll need to take $\textrm{gr}_D^{\textrm{orb}}(\bm{x_i})=0$ instead. We define the relative $\mathbb{Z}_2$-grading $\textrm{gr}^{\textrm{orb}}$ on the $\mathbb{Z}_2$-chain complex $\widehat{\mathit{CFO}}(Y^{\textrm{orb}}) = \widehat{\mathit{CFA}}(\mathcal{H}_E, \mathcal{Z}) \boxtimes D_N$ to be $\textrm{gr}_A + \textrm{gr}_D^{\textrm{orb}}$ (mod 2). Note that $\textrm{gr}^{\textrm{orb}}$ generalizes Equation \ref{Eqn51} because $\textrm{gr}_D^{\textrm{orb}}$ generalizes the relative $\mathbb{Z}_2$-grading $\textrm{gr}_D$ on $\widehat{\mathit{CFD}}(\mathcal{H}_{D^2 \times S^1})$ from Example \ref{regulargradingsolidtorus}. With respect to the induced relative $\mathbb{Z}_2$-grading $\textrm{gr}^{\textrm{orb}}$ on $\widehat{\mathit{HFO}}(Y^{\textrm{orb}})$, we have the following:

\begin{Lem}\label{Lem55}
\begin{equation}
  |\chi\big(\widehat{\mathit{HFO}}(Y^{\emph{orb}})\big)|=\begin{cases}
    n \cdot |H_1(|Y^{\emph{orb}}|)|, & \text{\emph{if} $|H_1(|Y^{\emph{orb}}|)|$ \emph{finite}}\\
    0, & \text{\emph{otherwise}}.
  \end{cases}
\end{equation}
\end{Lem}

\begin{proof}
For the 3-manifold $|Y^\textrm{orb}|$, we have
\begin{equation*}
  |\chi\big(\widehat{\mathit{CF}}(|Y^{\textrm{orb}}|)\big)|=\begin{cases}
     |H_1(|Y^{\textrm{orb}}|)|, & \textrm{\textrm{if} $|H_1(|Y^{\textrm{orb}}|)|$ \textrm{finite}}\\
    0, & \text{\textrm{otherwise}}.
  \end{cases}
\end{equation*}
Since $\chi\big(\widehat{\mathit{HFO}}(Y^{\textrm{orb}})\big) = \chi\big(\widehat{\mathit{CFO}}(Y^{\textrm{orb}})\big)$, it suffices to show that $|\chi\big(\widehat{\mathit{CFO}}(Y^{\textrm{orb}})\big)|= n \cdot |\chi\big(\widehat{\mathit{CF}}(|Y^{\textrm{orb}}|)\big)|$ when $|H_1(|Y^{\textrm{orb}}|)|$ finite. Note that for every $i$, $\bm{y} \otimes \bm{x_i} \neq 0$ exactly when $\bm{y} \otimes \bm{x} \neq 0$ and $\textrm{gr}^{\textrm{orb}}_D(\bm{x_i}) = \textrm{gr}_D(\bm{x})$. Then up to sign
\begin{equation*}
\begin{split}
\chi\big(\widehat{\mathit{CFO}}(Y^{\textrm{orb}})\big) = & \# \{ \bm{y} \otimes \bm{x_i} \neq 0 \mid \textrm{gr}_A(\bm{y}) + \textrm{gr}^{\textrm{orb}}_D(\bm{x_i}) = 0\}-\\
& \# \{ \bm{y} \otimes \bm{x_i} \neq 0 \mid \textrm{gr}_A(\bm{y}) + \textrm{gr}^{\textrm{orb}}_D(\bm{x_i}) = 1\}\\
= & n \Big( \# \{ \bm{y} \otimes \bm{x} \neq 0 \mid \textrm{gr}_A(\bm{y}) + \textrm{gr}_D(\bm{x}) = 0|\} -\\
& \# \{ \bm{y} \otimes \bm{x} \neq 0 \mid \textrm{gr}_A(\bm{y}) + \textrm{gr}_D(\bm{x}) = 1|\} \Big)\\
= & n \cdot \chi\big(\widehat{\mathit{CF}}(|Y^{\textrm{orb}}|)\big).
\end{split}
\end{equation*}
\end{proof}

Suppose $K$ is nullhomologous in $|Y^{\textrm{orb}}|$. We want to show
\begin{equation*}
|\chi\big(\widehat{\mathit{HFO}}(Y^{\textrm{orb}})\big)|=\begin{cases}
    |H_1^{\textrm{orb}}(Y^{\textrm{orb}})|, & \text{\textrm{if} $|H_1^\textrm{orb}(Y^{\textrm{orb}})|$ \textrm{finite}}\\
    0, & \text{\textrm{otherwise}}.
  \end{cases}
\end{equation*}
By \cite[Lemma 6.4]{Wong}, $H_1^{\textrm{orb}}(Y^{\textrm{orb}}) \cong H_1(|Y^{\textrm{orb}}|) \times \mathbb{Z}_n\langle \mu \rangle$, where $\mu$ is a meridian of $K$. This, combined with Lemma \ref{Lem55}, tells us that if $|H_1^\textrm{orb}(Y^{\textrm{orb}})|$ is finite, then $|\chi\big(\widehat{\mathit{HFO}}(Y^{\textrm{orb}})\big)|= n \cdot |H_1(|Y^{\textrm{orb}}|)| = |H_1^{\textrm{orb}}(Y^{\textrm{orb}})|$, as needed. Now suppose $|H_1^\textrm{orb}(Y^{\textrm{orb}})|$ is infinite. Then $|H_1(|Y^{\textrm{orb}}|)|$ is infinite, and by Lemma \ref{Lem55}, $|\chi\big(\widehat{\mathit{HFO}}(Y^{\textrm{orb}})\big)| =0$. This concludes the proof of Theorem \ref{categorification} for $K$ nullhomologous in $|Y^{\textrm{orb}}|$.

Now let $Y$ be $\frac{p}{q}$-surgery on a knot $K \subset S^3$ with $\textrm{gcd}(p,q) =1$ and $p \geq 0$. Let $Y^{\textrm{orb}}$ be the 3-orbifold with underlying space $Y$ and singular curve $K$ of multiplicity $n$. Again we want to show
\begin{equation*}
|\chi\big(\widehat{\mathit{HFO}}(Y^{\textrm{orb}})\big)|=\begin{cases}
    |H_1^{\textrm{orb}}(Y^{\textrm{orb}})|, & \text{\textrm{if} $|H_1^\textrm{orb}(Y^{\textrm{orb}})|$ \textrm{finite}}\\
    0, & \text{\textrm{otherwise}}.
  \end{cases}
\end{equation*}
It's not hard to see that $H_1(Y) \cong \mathbb{Z}_{p} \langle \mu \rangle$ and $H_1^{\textrm{orb}}(Y^{\textrm{orb}}) \cong \mathbb{Z}_{np} \langle \mu \rangle$, where $\mu$ is a meridian of $K$. We again have two cases. First suppose $|H_1^\textrm{orb}(Y^{\textrm{orb}})|$ is finite. Then $p \neq 0$ and $|H_1^\textrm{orb}(Y^{\textrm{orb}})| = n \cdot |H_1(Y)|$. From Lemma \ref{Lem55}, $|\chi\big(\widehat{\mathit{HFO}}(Y^{\textrm{orb}})\big)| = n \cdot |H_1(Y)| = |H_1^\textrm{orb}(Y^{\textrm{orb}})|$, as desired. Now suppose $|H_1^\textrm{orb}(Y^{\textrm{orb}})|$ is infinite. Then $p=0$, which means $|H_1(Y)|$ is infinite. Again by Lemma \ref{Lem55}, $|\chi\big(\widehat{\mathit{HFO}}(Y^{\textrm{orb}})\big)| =0$. This concludes the proof of Theorem \ref{categorification} for $Y^{\textrm{orb}} = (Y, K, n)$, where $Y$ is $\frac{p}{q}$-surgery on a knot $K \subset S^3$.

{\qed} \popQED

\bibliography{OrbifoldsHF}

\begin{thebibliography}{10}

\bibitem{BMP}
Michel Boileau, Sylvain Maillot, and Joan Porti.
\newblock {\em Three-dimensional orbifolds and their geometric structures},
  volume~15 of {\em Panoramas et Synth\`eses}.
\newblock Soci\'et\'e Math\'ematique de France, Paris, 2003.

\bibitem{CS}
O.~Collin and B.~Steer.
\newblock Instanton {F}loer homology for knots via {$3$}-orbifolds.
\newblock {\em J. Differential Geom.}, 51(1):149--202, 1999.

\bibitem{Floer}
A.~Floer.
\newblock An instanton-invariant for {$3$}-manifolds.
\newblock {\em Comm. Math. Phys.}, 118(2):215--240, 1988.

\bibitem{Greene}
J.~Greene.
\newblock The lens space realization problem.
\newblock {\em Ann. of Math. (2)}, 177(2):449--511, 2013.

\bibitem{HRW}
J.~Hanselman, R.~Rasmussen, and L.~Watson.
\newblock Bordered {F}loer homology for manifolds with torus boundary via
  immersed curves.
\newblock {\em arXiv:1604.03466v2}, 2017.

\bibitem{HL}
M.~Hedden and A.~Levine.
\newblock Splicing knot complements and bordered {F}loer homology.
\newblock {\em J. Reine Angew. Math.}, 720:129--154, 2016.

\bibitem{Hom}
J.~Hom.
\newblock Bordered {H}eegaard {F}loer homology and the tau-invariant of cable
  knots.
\newblock {\em J. Topol.}, 7(2):287--326, 2014.

\bibitem{HLW}
Jennifer Hom, Tye Lidman, and Liam Watson.
\newblock The {A}lexander module, {S}eifert forms, and categorification.
\newblock {\em J. Topol.}, 10(1):22--100, 2017.

\bibitem{KL}
B.~Kleiner and J.~Lott.
\newblock Geometrization of three-dimensional orbifolds {\it via} {R}icci flow.
\newblock {\em Ast\'erisque}, (365):101--177, 2014.

\bibitem{KM2}
P.~Kronheimer and T.~Mrowka.
\newblock Khovanov homology is an unknot-detector.
\newblock {\em Publ. Math. Inst. Hautes \'Etudes Sci.}, (113):97--208, 2011.

\bibitem{KM1}
P.~Kronheimer and T.~Mrowka.
\newblock Knot homology groups from instantons.
\newblock {\em J. Topol.}, 4(4):835--918, 2011.

\bibitem{KM3}
P.~Kronheimer and T.~Mrowka.
\newblock Tait colorings, and an instanton homology for webs and foams.
\newblock {\em arXiv:1508.07205v1}, 2015.

\bibitem{KMVW}
C.~Kutluhan, G.~Mati\'c, J.~Van Horn-Morris, and A.~Wand.
\newblock Filtering the {H}eegaard {F}loer contact invariant.
\newblock {\em arXiv:1603.02673v4}, 2018.

\bibitem{Levine}
Adam~Simon Levine.
\newblock Knot doubling operators and bordered {H}eegaard {F}loer homology.
\newblock {\em J. Topol.}, 5(3):651--712, 2012.

\bibitem{LOT2}
R.~Lipshitz, P.~Ozsv\'ath, and D.~Thurston.
\newblock Bimodules in bordered {H}eegaard {F}loer homology.
\newblock {\em Geom. Topol.}, 19(2):525--724, 2015.

\bibitem{LOT1}
R.~Lipshitz, P.~Ozsv\'ath, and D.~Thurston.
\newblock Bordered {H}eegaard {F}loer homology: Invariance and pairing.
\newblock {\em Memoirs of the American Mathematical Society}, to appear.

\bibitem{Ni}
Y.~Ni.
\newblock Knot {F}loer homology detects fibred knots.
\newblock {\em Invent. Math.}, 170(3):577--608, 2007.

\bibitem{OS6}
P.~Ozsv\'ath, A.~Stipsicz, and Z.~Szab\'o.
\newblock Concordance homomorphisms from knot {F}loer homology.
\newblock {\em Adv. Math.}, 315:366--426, 2017.

\bibitem{OS4}
P.~Ozsv\'ath and Z.~Szab\'o.
\newblock Knot {F}loer homology and the four-ball genus.
\newblock {\em Geom. Topol.}, 7:615--639, 2003.

\bibitem{OS3}
P.~Ozsv\'ath and Z.~Szab\'o.
\newblock Holomorphic disks and genus bounds.
\newblock {\em Geom. Topol.}, 8:311--334, 2004.

\bibitem{OS8}
P.~Ozsv\'ath and Z.~Szab\'o.
\newblock Holomorphic disks and knot invariants.
\newblock {\em Adv. Math.}, 186(1):58--116, 2004.

\bibitem{OS2}
P.~Ozsv\'ath and Z.~Szab\'o.
\newblock Holomorphic disks and three-manifold invariants: properties and
  applications.
\newblock {\em Ann. of Math. (2)}, 159(3):1159--1245, 2004.

\bibitem{OS1}
P.~Ozsv\'ath and Z.~Szab\'o.
\newblock Holomorphic disks and topological invariants for closed
  three-manifolds.
\newblock {\em Ann. of Math. (2)}, 159(3):1027--1158, 2004.

\bibitem{OS5}
P.~Ozsv\'ath and Z.~Szab\'o.
\newblock Heegaard {F}loer homology and contact structures.
\newblock {\em Duke Math. J.}, 129(1):39--61, 2005.

\bibitem{OS7}
P.~Ozsv\'ath and Z.~Szab\'o.
\newblock Knots with unknotting number one and {H}eegaard {F}loer homology.
\newblock {\em Topology}, 44(4):705--745, 2005.

\bibitem{OS10}
Peter~S. Ozsv\'ath and Zolt\'an Szab\'o.
\newblock Knot {F}loer homology and rational surgeries.
\newblock {\em Algebr. Geom. Topol.}, 11(1):1--68, 2011.

\bibitem{Ras}
J.~Rasmussen.
\newblock {\em Floer homology and knot complements}.
\newblock PhD thesis, Harvard University, 2003.

\bibitem{Scott}
P.~Scott.
\newblock The geometries of {$3$}-manifolds.
\newblock {\em Bull. London Math. Soc.}, 15(5):401--487, 1983.

\bibitem{Thurston}
W.~Thurston.
\newblock {\em Three-dimensional geometry and topology. {V}ol. 1}, volume~35 of
  {\em Princeton Mathematical Series}.
\newblock Princeton University Press, Princeton, NJ, 1997.

\bibitem{Wong}
B.~Wong.
\newblock Turaev torsion invariants of 3-orbifolds.
\newblock {\em Geom. Dedicata}, 187:179--197, 2017.

\end{thebibliography}
\bibliographystyle{plain}

\end{document}